\documentclass{amsart}
\usepackage[utf8x]{inputenc}
\usepackage{amsmath, amsthm, amssymb}
\usepackage[usenames,dvipsnames,svgnames,table]{xcolor}
\usepackage[margin=1.3in]{geometry}
\usepackage{esint,dsfont}
\usepackage{cleveref}
\usepackage[normalem]{ulem}
\usepackage{enumerate}

\newtheorem{theorem}{Theorem}[section]
\newtheorem{proposition}[theorem]{Proposition}
\newtheorem{lemma}[theorem]{Lemma}
\newtheorem{definition}[theorem]{Definition}
\newtheorem*{remark}{Remark}

\newcommand{\E}{\mathbb E}
\newcommand{\N}{\mathbb N}
\newcommand{\PP}{\mathbb P}
\newcommand{\R}{\mathbb R}
\newcommand{\Q}{\mathbb Q}

\newcommand{\eps}{\varepsilon}
\newcommand{\e}{\eps}

\newcommand{\dd}{\, \mathrm{d}}

\newcommand{\tr}{\mbox{tr}}
\newcommand{\1}{\mathds{1}}
\newcommand{\vv}{\langle v\rangle}
\newcommand{\ul}{\rm ul}

\DeclareMathOperator{\supp}{supp}

\numberwithin{equation}{section}

\title[Existence and lower mass bounds for the Landau equation]{Local existence, lower mass bounds, and a new continuation criterion for the Landau equation}
\author{Christopher Henderson}
\address{Department of Mathematics, University of Chicago, 5734 S. University Ave., Chicago, IL 60637}
\email{henderson@math.uchicago.edu}
\author{Stanley Snelson}
\address{Department of Mathematical Sciences, Florida Institute of Technology, 150 W. University Blvd., Melbourne, FL 32901}
\email{ssnelson@fit.edu}
\author{Andrei Tarfulea}
\address{Department of Mathematics, University of Chicago, 5734 S. University Ave., Chicago, IL 60637}
\email{atarfulea@math.uchicago.edu}
\thanks{All three authors were partially supported by NSF grant DMS-1246999.}

\begin{document}

\begin{abstract}
We consider the spatially inhomogeneous Landau equation with soft potentials, including the case of Coulomb interactions. First, we establish the existence of solutions for a short time, assuming the initial data is in a fourth-order Sobolev space and has Gaussian decay in the velocity variable (no decay assumptions are made in the spatial variable). Next, we show that the evolution instantaneously spreads mass to every point in its domain. The resulting pointwise lower bounds have a sub-Gaussian rate of decay, which we show is optimal. 
The proof of mass-spreading is based on a stochastic process associated to the equation, and makes essential use of nonlocality. By combining  this 
theorem with prior regularity results, we derive two important applications: $C^\infty$ smoothing in all three variables, even for initial data with vacuum regions, and a continuation criterion that states the solution can be extended for as long as the mass and energy densities stay bounded from above. This is the weakest condition known to prevent blow-up.  In particular, it does not require a lower bound on the mass density or an upper bound on the entropy density.
\end{abstract}

\maketitle

\section{Introduction}


\subsection{The Landau equation} We are interested in the Landau equation from plasma physics, in three space dimensions:
\begin{equation}\label{e:divergence}
\partial_t f + v\cdot \nabla_x f = Q_L(f,f) = \tr\left(\bar a[f]D_v^2 f\right) + \bar c[f] f.
\end{equation}
For a constant $\gamma >-3$,  the coefficients $\bar a[f](t,x,v)\in \R^{3\times 3}$ and $\bar c[f](t,x,v)\in \R$ are defined by
\begin{align}
\bar a[f] &:= a_{\gamma}\int_{\R^3} \left( I - \frac w {|w|} \otimes \frac w {|w|}\right) |w|^{\gamma + 2} f(t,x,v-w) \dd w,\label{e:a}\\
\bar c[f] &:= c_{\gamma}\int_{\R^3} |w|^\gamma f(t,x,v-w)\dd w, \label{e:c}
\end{align}
where $a_{\gamma}, c_\gamma>0$ are constants and $I$ is the identity matrix on $\R^3$. For $\gamma = -3$, the expression of $\bar a[f]$ is unchanged and the expression for $\bar c[f]$ must be replaced by $c_3 f$ for a fixed constant $c_3>0$.  We assume that initial data $f_{in}(x,v) \geq 0$ is given.   
 The solution $f(t,x,v)\geq 0$ to \eqref{e:divergence} models the evolution of a particle density in phase space. The Landau equation arises as the limit of the Boltzmann equation as grazing collisions predominate. See, for example,  \cite{chapmancowling, lifschitzpitaevskii} for the physical background. We are interested in the case of \emph{soft potentials}, i.e. $\gamma \in [-3,0)$. The case $\gamma \in [-2,0)$ is also called moderately soft potentials, $\gamma \in (-3,-2)$ is called very soft potentials, and $\gamma=-3$ is called the Landau-Coulomb equation.

It is not currently known whether global-in-time classical solutions to \eqref{e:divergence} exist for general, non-perturbative initial data. In this paper, we establish the existence of a $C^\infty$ solution of the Cauchy problem for \eqref{e:divergence} on some time interval $[0,T]$, without a smallness assumption on the initial data, and give a continuation criterion in terms of physically relevant quantities. Define
\begin{align*}
M(t,x) &= \int_{\R^3} f(t,x,v)\dd v, &&\mbox{(mass density)}\\
E(t,x) &= \int_{\R^3} |v|^2 f(t,x,v)\dd v, &&\mbox{(energy density)}\\
H(t,x) &= \int_{\R^3} f(t,x,v)\log f(t,x,v) \dd v. &&\mbox{(entropy density)}
\end{align*}
In the homogeneous setting, i.e.~when there is no $x$ dependence, the mass and energy of the solution are conserved and the entropy is decreasing, but these properties have not been shown for the full inhomogeneous equation.  Several regularity results for weak solutions of \eqref{e:divergence} have been derived under the assumption that $M(t,x)$, $E(t,x)$, and $H(t,x)$ are bounded above, and that $M(t,x)$ is bounded away from zero, including  \cite{cameron2017landau, golse2016, henderson2017smoothing}. (See the Related Work section below.)  The present paper makes  three 
contributions: first, to develop a local existence theory that is compatible with these \emph{a priori} results; second, to remove the assumptions that the mass of $f$ is bounded from below and the entropy is bounded from above 
from the regularity criteria for our solution; third, to show that solutions may be continued so long as the mass and energy can be controlled from above. 

In both the Boltzmann and Landau equations, coercivity of the bilinear collision operator is a crucial ingredient in proving regularization theorems. For the Landau equation, the coercivity of $Q_L(f,f)$ comes from the ellipticity of the matrix $\bar a[f]$, and it is clear from \eqref{e:a} that the ellipticity may degenerate if $M(t,x) = 0$ at some $(t,x)$. 
Our Theorem \ref{thm:mass} implies that, under relatively weak assumptions, $M(t,x)$ is necessarily positive for $t>0$ and $\bar a[f]$ is uniformly elliptic. Our method also allows us to remove the entropy bound from our criteria for smoothness and continuation.

\subsection{Main results} Before stating our results, we define the uniformly local weighted Sobolev spaces we use. Uniformly local spaces (first introduced by Kato in \cite{kato1975hyperbolic}) do not specify any decay as $|x|\to \infty$. Let $\vv = (1+|v|^2)^{1/2}$. For integers $k,\ell \geq 0$, define
\[\|g\|_{H^{k,\ell}_{\ul}(\R^6)} = \sum_{|\alpha|+|\beta|\leq k} \sup_{a\in \R^3}  \int_{\R^6} |\phi(x-a) \vv^\ell \partial_x^\alpha \partial_v^\beta g(x,v)|^2 \dd x \dd v,\]
where $\phi\in C^\infty_0(\R^3)$ is a cutoff satisfying $0\leq \phi \leq 1$, $\phi \equiv 1$ in $B_1$, and $\phi\equiv 0$ in $\R^3\setminus B_2$.  Here, for $\alpha, \beta \in (\N\cup\{0\})^3$, $\partial_x^\alpha = \partial_{x_1}^{\alpha_1}\partial_{x_2}^{\alpha_2}\partial_{x_3}^{\alpha_3}$.  We define $\partial_v^\beta$ similarly. Then, our uniformly local spaces are:
\begin{equation}\label{e:ul-def}
H^{k,\ell}_{\ul}(\R^6) = \{ g : \|g\|_{H^{k,\ell}_{\ul}(\R^6)} < \infty\}.
\end{equation}
We often write $H^{k,\ell}_{\ul} = H^{k,\ell}_{\ul}(\R^6)$. Let $H^{k}_{\ul} = H^{k,0}_{\ul}$, $L^{2,\ell}_{\ul}= H^{0,\ell}_{\ul}$, and $L^2_{\ul} = L^{2,0}_{\ul}$. Our main time-dependent spaces will be 
\begin{equation}\label{eq:Y_definition} 
Y^k_T:= L^\infty([0,T],H^{k}_{\ul}(\R^6)) \cap L^2([0,T], H^{k,1}_{\ul}(\R^6)).
\end{equation}
 
Our first main result is the local well-posedness of \eqref{e:divergence}:
\begin{theorem}\label{t:LWP}
Let $f_{in}(x,v):\R^3\times\R^3\to\R_+$ be such that $e^{\rho_0 \vv^2}f_{in}(x,v) \in H^k_{\ul}(\R^6)$ for some $\rho_0>0$ and some $k\geq 4$. Then for some $T>0$, depending only on $\gamma$, $\rho_0$, $k$, and $\|e^{\rho_0 \vv^2} f_{in}\|_{H^k_{\ul}}$, there is a unique solution $f\geq 0$ to \eqref{e:divergence} with $f(0,x,v) = f_{in}(x,v)$ and $e^{\rho_0 \vv^2/2} f \in Y_T^k \cap C^0([0,T], H^k_{\ul}(\R^6))$.
\end{theorem}

While the proof of Theorem 1.1 involves mostly classical techniques, a
local existence result for the Landau equation was previously missing from
the literature (except in the case $\gamma = -3$, see \cite{he2014boltzmannlandau}), and we
believe it is important to fill this gap. Our existence theorem also
provides a convenient framework for our other results.

Note that, in light of the definition of $H_{\ul}^{k,l}$, Theorem 1.1 makes
no assumption on the behavior of $f_{in}$ as $|x| \to \infty$. The
requirement that $f_{in}$ have four Sobolev derivatives is an improvement
over \cite{he2014boltzmannlandau}, and matches the current state of the art for results on the
Boltzmann equation \cite{amuxy2010regularizing, amuxy2011bounded, amuxy2013mild}. At this time, it is unclear whether this
hypothesis can be relaxed further.

Next, by relating equation \eqref{e:divergence} to a certain stochastic process,
we prove that the mass density $M(t,x)$ instantly becomes positive,
and moreover stays uniformly positive on compact sets away from $t=0$.  Before stating this theorem, we define well-distributed initial data, a hypothesis under which we can strengthen our results.
\begin{definition}\label{d:well}
We say that a function $g: \R^3 \times \R^3 \to [0,\infty)$ is \emph{well-distributed with parameters $R, \delta, r>0$} if, for every $x \in \R^3$, there exists $x_m \in B_{R}(x)$ and $v_m \in B_R(0)$ such that $g \geq \delta \1_{B_r(x_m)\times B_r(v_m)}$.
\end{definition}
Heuristically, a function is well-distributed if for every $x \in \R^3$, there is some uniform amount of mass nearby at relatively low velocities.  For a simple example of well-distributed initial data consider the following: if $\chi_0$ and $\chi_{\rm per}$ are continuous, non-negative, non-zero functions, with $\chi_{\rm per}$ periodic, then any $f_0(x,v) \geq \chi_0(v) \chi_{\rm per}(x)$ is well-distributed.

Now we state our mass-pushing theorem, which will be crucial in obtaining the smoothness of solutions to the Landau equation. For technical reasons, we work with solutions of the type constructed in Theorem \ref{t:LWP}, but a similar property should be expected to hold for solutions with weaker decay and regularity. 

\begin{theorem}\label{thm:mass}
Suppose $0<\underline T \leq T \leq \overline T$ and that $f$ is a solution of \eqref{e:divergence} on $[0,T]$ such that $e^{\rho\vv^2}f \in Y^4_T$ for some $\rho >0$, and such that $f$ is non-negative and not uniformly equal to zero.
\begin{enumerate}[(i)]
	\item Then, for any $(t,x) \in [\underline T, T]\times \R^3$, there exists $\nu_{\underline T, x}, \rho_{\underline T,x}>0$ depending only on $\underline T$, $\overline T$, $K$, the initial data, and the quantities
\begin{equation}\label{eq:physical}\begin{cases}
	\displaystyle\sup_{x\in \R^3, t\in [0,T]} \left(M(t,x)+E(t,x)\right), &\hspace{-.2 in} \text{ if } \gamma \in (-2,0), \text{ or }\\
	\displaystyle\sup_{x\in \R^3, t\in [0,T]} \left(M(t,x) + E(t,x)  + P(t,x) + \|f(t,x,\cdot)\|_{L^\infty(\R^6)}\right), & \text{ if }\gamma \in [-3,-2],
\end{cases}
\end{equation} 
such that, for all $v\in \R^3$,
\[
		f(t,x,v) \geq \nu_{\underline T,x} \exp \left\{ - \rho_{\underline T,x} |v|^{\max\{3-\gamma,4\}}\right\}.
\]
Here, $P(t,x) = \int_{\R^3}|v|^p f(t,x,v)\dd v$ with $p > 3|\gamma| / (5+\gamma)$.
\item If, in addition, 
$f(0,\cdot,\cdot)$ is well-distributed for some parameters $R$, $\delta$, $r$, then, we find $\nu_{\underline T}, \rho>0$ depending on $R$, $\delta$, $r$, $\overline T$, $K$, and the quantities~\eqref{eq:physical}, with $\nu_{\underline T}$ depending additionally on $\underline T$, such that, for all $(t,x,v) \in [\underline T, T]\times \R^3 \times \R^3$
\begin{equation}\label{eq:Gaussian}
	f(t,x,v) \geq \nu_{\underline T} \exp\left\{ - \rho |v|^{2-\gamma}\right\}.
\end{equation}
\end{enumerate}
\end{theorem}

In the well-distributed case of~\eqref{eq:Gaussian}, one might expect a lower bound like $e^{-\rho|v|^2}$, as these ``Maxwellians'' are the steady state solutions of the Landau equation.  The above result, however, is sharp for well-distributed initial data, as we demonstrate in Proposition \ref{p:upper_Gaussian}. 
In light of results that show convergence to Maxwellians for \emph{a priori} global solutions such as \cite{desvillettes2005trend}, or for solutions starting close to equilibrium such as \cite{mouhot2006equilibrium, strain2008exponential}, we infer that the comparatively fatter tails of the Maxwellians form as $t\to\infty$. 
For non-well-distributed initial data, it is not clear whether the lower bounds in Theorem \ref{thm:mass}(i) are optimal. We leave this question for future work.

The proof of Theorem \ref{thm:mass} relies on probabilistic methods to show that a positive amount of mass anywhere in space and velocity can spread (at least a small amount) to any other location and velocity instantaneously. Given the kinetic setting, which naturally involves following ``random'' trajectories, probabilistic methods seem well-adapted to the problem, and there is a somewhat rich history of using stochastic processes to study kinetic equations (see Section \ref{s:prob} below).

Theorem \ref{thm:mass} implies in particular that $M(t,x) = \int_{\R^3} f(t,x,v) \dd v >0$ for every $t$ and $x$, that the positive lower bound is uniform locally in $t>0$ and $x\in\R^3$, and that it is uniform in $x$, for fixed $t$, when the initial data is well-distributed.  
Theorem \ref{thm:mass} also implies 
$\bar a[f]$ is uniformly elliptic.  An almost immediate consequence of this, along with the Schauder estimates of~\cite{henderson2017smoothing}, is the smoothness of $f$, as stated in the following:
\begin{theorem}\label{thm:smooth}
Suppose $0 < \underline T \leq T \leq \overline T$ and that $f$ solves \eqref{e:divergence} with $e^{\rho\vv^2} f \in Y_T^4$ for some $\rho >0$. Then $f \in C^\infty((0,T]\times \R^3\times \R^3)$, and for every $(t,x)\in [\underline T, T]\times \R^3$, we have $\|\vv^\ell f(t,x,\cdot)\|_{H^k(\R^3)} \leq C_{k,\ell,t,x}$ for all integers $k,\ell\geq 0$.  If the initial data is well-distributed, then $\|\vv^\ell f \|_{L^\infty([\underline T,T], H^k(\R^6))} \leq C_{k,\ell, \underline T,\overline T}$ for all integers $k,\ell \geq 0$.
\end{theorem}
\Cref{thm:smooth} applies in particular to the solution constructed in \Cref{t:LWP}.  To our knowledge, this is the first $C^\infty$ solution to the Cauchy problem for \eqref{e:divergence} that does not require perturbative initial data.  Also, note that the initial data of our solution may contain vacuum regions.

Finally, we show that our solution can be extended as long as the physical quantities remain bounded above.  In particular, this implies that any blow-up of solutions to the Landau equation with suitable initial data must occur at the level of the quantities~\eqref{eq:physical}.  We state this roughly at present, for the ease of the reader (see \Cref{t:extension2} for a more detailed statement).
\begin{theorem}\label{t:extension}

If the initial data $f_{in}$ is well-distributed and $e^{\rho_0 \vv^2} f_{in} \in H^4_{ul}$ for some $\rho_0 >0$, then a unique solution to \eqref{e:divergence} exists for as long as the quantity \eqref{eq:physical} remains finite.
\end{theorem}

In the case $\gamma \in (-2,0)$, Theorem \ref{t:extension} gives a physically meaningful continuation criterion.  Namely, that blow-up can be prevented, and the solution extended, by obtaining upper bounds on the mass and energy.  When $\gamma \in[-3,-2]$, one must control also the higher moment $P$ and the $L^\infty$ norm of $f$. The additional restrictions in \Cref{thm:mass,t:extension} in the case $\gamma \in [-3,-2]$ (see~\eqref{eq:physical}) are inherited from~\cite{henderson2017smoothing}.  This is related to the challenging issue of finding an \emph{a priori} $L^\infty$ bound for $f$ when $\gamma \leq -2$, which is open even in the spatially homogeneous case.

The requirement that the quantities \eqref{eq:physical} remain finite is the weakest known condition for global existence of solutions to the inhomogeneous Landau equation. 
We emphasize that we do not require an \emph{a priori} positive lower bound on the mass density $M(t,x)  = \int_{\R^3} f(t,x,v)\dd v$, as is required in \cite{cameron2017landau, golse2016, henderson2017smoothing, silvestre2015landau}.  While earlier regularity results for the Landau equation such as \cite{golse2016, henderson2017smoothing} were based on adapting the corresponding theory for local equations, our proof of Theorem \ref{thm:mass} makes essential use of nonlocality, since the spreading of mass from $(t_0,x_0,v_0)$ to $(t_1,x_1,v_1)$ relies on velocities that are in general far from $v_0$ or $v_1$.



\subsection{Related work}

%
%

So far, global-in-time classical solutions to \eqref{e:divergence} have only been constructed for initial data close to an equilibrium state: see the work of Guo \cite{guo2002periodic} in the $x$-periodic case and Mouhot-Neumann \cite{mouhot2006equilibrium} with $x\in \R^3$. For general initial data, Villani \cite{villani1996global} showed the existence of renormalized solutions with defect measure. 
Existence or non-existence of classical global-in-time solutions 
for general initial data remains a challenging open problem.
Regarding short-time existence, spatially periodic classical solutions were found by He-Yang \cite{he2014boltzmannlandau} in the Coulomb case ($\gamma = -3$) by taking the grazing collisions limit in their estimates on the Boltzmann collision operator. They take initial data in a weighted $H^7_{x,v}$ space, with mass density bounded away from zero. Compared to \cite{he2014boltzmannlandau}, the present paper makes a stronger decay assumption on $f_{in}$ in $v$, but improves on the required number of derivatives, covers both the cases $\gamma = -3$ and $\gamma \in (-3,0)$, 
and allows $f_{in}$ to have vacuum regions, which is more satisfactory on physical grounds. 


There is a previous ``mass pushing'' result for the Boltzmann equation, due to Briant~\cite{Briant}, which is obtained on $\mathbb{T}^3 \times \R^3$ or $\Omega\times \R^3$, where $\Omega$ is a smooth, convex domain. It is shown that vacuum regions are immediately filled and the solution obeys a lower bound of the form $f(t,x,v) \gtrsim \exp\{-|v|^K/C\}$ for some $C$ and some explicit $K$. However, the methods of~\cite{Briant} leverage the fact that the differential operator in the Boltzmann equation is an integral operator; this is advantageous for obtaining pointwise bounds, but is unavailable for the Landau equation.   Further, the lower bound of~\cite{Briant} depends on stronger norms of $f$ and the constant $K$ tends to infinity in the grazing collisions limit that recovers Landau from Boltzmann.  As such, our proof is completely independent of that in~\cite{Briant}.  We mention also the earlier work of~\cite{Mouhot, PulvirentiWennberg}, on which~\cite{Briant} is based.  We note the relative simplicity of our proof when compared to that of~\cite{Briant}.

Our regularity results make use of prior work from the last few years on weak solutions of \eqref{e:divergence} with $M(t,x)$, $E(t,x)$, and $H(t,x)$ bounded above, and $M(t,x)$ bounded below. In that context,  Golse-Imbert-Mouhot-Vasseur \cite{golse2016} showed local H\"older continuity (see also Wang-Zhang \cite{Wang2011ultraparabolic}), Cameron-Silvestre-Snelson \cite{cameron2017landau} showed global H\"older continuity and propagation of Gaussian bounds (in the case $\gamma \in (-2,0)$), and Henderson-Snelson \cite{henderson2017smoothing} established $C^\infty$ regularity for $\gamma <0$, with stronger assumptions on $f$ in the case $\gamma \leq -2$. At least for solutions in the class we consider, Theorem \ref{thm:mass} allows us to improve this regularity criterion. Earlier smoothing results for \eqref{e:divergence} such as \cite{chen2009smoothing, liu2014regularization} that make much stronger assumptions on $f$ (infinitely many moments in $v$ bounded in $H^8_{x,v}$) also include a condition about $f$ having mass bounded below, either explicitly or as part of the assumption that $f$ is close to a Maxwellian equilibrium. The same is true of many smoothing results for the Boltzmann equation such as \cite{amuxy2009regularity, amuxy2010regularizing, chen2012smoothing}. 

There has also been a great deal of work on existence and regularity for the spatially homogeneous Landau equation, which results from taking $f$ independent of $x$ in \eqref{e:divergence}. We refer to  \cite{alexandre2015apriori, desvillettes2015landau, desvillettes2000landau, gualdani2014radial, gualdani2017landau, silvestre2015landau, villani1998landau, wu2014global} and the references therein.

\subsection{Probabalistic approaches to kinetic equations}\label{s:prob}

Early work related the homogeneous Boltzmann and Landau equation to a fully nonlinear stochastic process which, through Malliavin calculus or Martingale theory, could recover weak solutions to the equations (see \cite{Tanaka1979, Funaki1985}). By relating these processes to certain Wasserstein distances, the weak function-solutions could be shown to be unique (see \cite{TV1999, FG2008Boltzmann, FG2009Landau} and references therein); i.e., the distance between two weak function-solutions is nonincreasing along the flow generated by the equation. The techniques have since been adapted to show higher regularity (with some a priori assumptions) for solutions to the homogeneous equations for the case of Maxwell molecules (see for instance \cite{Guerin2002, GMN2006, Fournier2008, DMN2015}), though the techniques can also apply to moderately soft potentials \cite{Guerin2003}. These approaches are limited to the homogeneous (and largely measure-valued) setting because they relate the equations to a fully nonlinear stochastic process, which is then used to build the solutions. To the best of our knowledge, Theorem \ref{thm:mass} is the first application of ideas in probability to the mass distribution for the {\em inhomogeneous} Landau equation. The crucial difference is that we know, from Theorem \ref{t:LWP}, that a unique solution $f$ already exists; and, moreover, that it is H\"{o}lder continuous. For the proof of Theorem \ref{thm:mass}, we only need to relate the linearized Landau equation to a much simpler process (see Lemma \ref{lem:probabilistic}). Powerful pre-existing techniques are then applied to obtain a much shorter (and more precise) proof.


\subsection{Proof ideas}
The strategy of our proof of Theorem \ref{t:LWP} is as follows. First, we divide $f$ by a Gaussian with time-dependent decay. The equation \eqref{e:g} for the resulting function $g$ is approximated in multiple steps: we first solve a linearized version of the equation on a bounded domain with an extra diffusive term (Lemma \ref{l:parabolic}). By deriving appropriate uniform estimates, we use a compactness argument to take the limit as the size of the domain increases and the added diffusion goes to zero to find a solution of the linearized Landau equation on the whole space (Lemma \ref{l:gre}). Finally, we solve \eqref{e:g} by iteration, making use of our estimates for the linearized problem. This method is in some ways inspired by previous work on local well-posedness for the non-cutoff Boltzmann equation by the AMUXY group (Alexandre, Morimoto, Ukai, Xu and Yang), see \cite{amuxy2010regularizing, amuxy2011bounded, amuxy2013mild}. Those papers use an approximation scheme based on cutting off the angular singularity in the Boltzmann collision kernel.  Such an approximation is not available for the Landau equation because the Landau equation results from focusing on grazing collisions in the Boltzmann equation, i.e.~taking the limit where the angular singularity essentially becomes a derivative in $v$. 
We point out that our proof covers all cases $\gamma \in [-3,0)$, which requires extra care, while \cite{amuxy2011bounded, amuxy2013mild} make the restriction that $\gamma >-3/2$, and \cite{amuxy2010regularizing} replaces the factor $|v-w|^\gamma$ in the Boltzmann collision kernel with $(1+|v-w|^2)^{\gamma/2}$, which also sidesteps the difficulties associated with very soft potentials.

 Roughly speaking, mass spreading (Theorem \ref{thm:mass}) holds because nonzero initial data $f_{in}$ must have a ``core'' of positive mass near some point $(x_0,v_0)$, which spreads out instantaneously in $v$ because of the diffusive property of the equation, and some small amount of this mass is in turn spread out to any point $x$ at any time $t$ because of the pure advective term. By relating the value of $f$ to the expectation of a random variable (Lemma \ref{lem:probabilistic}), we show these properties by analyzing the associated stochastic process.   Here it is important to understand the trajectories along which the equation propagates information.  This allows us to roughly estimate how the process spreads mass from one point to another in $\R^3\times \R^3$.
 This mass-spreading leads almost immediately to \Cref{thm:smooth}, as mentioned above.  

To prove Theorem \ref{t:extension}, we need to apply the main theorem of \cite{henderson2017smoothing}, which states that weak solutions of \eqref{e:divergence} with Gaussian-decaying initial data are smooth for all $t>0$ provided $M(t,x)$, $E(t,x)$, and $H(t,x)$ are bounded above, and $M(t,x)$ is bounded below. With Lemma \ref{l:lower}, we can derive lower ellipticity constants for $f$ directly from the lower bounds of \Cref{thm:mass}, which allows us to side-step the conditions that $M(t,x)$ is bounded below and $H(t,x)$ is bounded above. 
Combining the estimates from \cite{henderson2017smoothing} with the results in \cite{cameron2017landau}, we obtain a Gaussian bound on $f$ at time $T$.  Applying \Cref{t:LWP} with initial data $f(T,\cdot,\cdot)$ provides the extension. Here it is crucial that the bounds obtained in \Cref{thm:mass} depend only on those quantities in~\eqref{eq:physical}.  

\subsection{Organization of the paper}  In Section \ref{s:prelim}, we establish various estimates on the coefficients of the equation that will be needed in the proof of existence. In Section \ref{s:LWP}, we prove local-in-time existence for $f$ (Theorem \ref{t:LWP}), and in Section \ref{s:mass}, we establish the mass-spreading property of the equation along with the sub-Gaussian  bounds (Theorem \ref{thm:mass}). In Section \ref{s:app}, we apply Theorem \ref{thm:mass} to show that our solution to \eqref{e:divergence} is $C^\infty$ (Theorem \ref{thm:smooth}) and that the solution can be extended for as long as the quantities \eqref{eq:physical} remain bounded (Theorem \ref{t:extension}). 

\section{Preliminaries}\label{s:prelim}

First, we introduce the following modified Cauchy problem: for $\rho_0, \kappa >0$, let $T_{\rho_0,\kappa} = \rho_0/(2\kappa)$, $\mu(t,v) = e^{-(\rho_0-\kappa t)\vv^2}$ and $g(t,x,v) = \mu(t,v)^{-1} f(t,x,v)$. From \eqref{e:divergence}, the equation for $g$ is
\begin{align}\label{e:g}
\partial_t g + v\cdot \nabla_x g + \kappa \vv^2 g &= \mu^{-1} Q_L(\mu g, \mu g)\nonumber\\
&= \mu^{-1}\tr \left( \bar a[\mu g](\mu D_v^2 g + g D_v^2 \mu + \nabla \mu \otimes \nabla g + \nabla g \otimes \nabla \mu)\right)    + \bar c[\mu g] g\nonumber\\
&= \tr (A[g] D_v^2 g) + B[g]\cdot \nabla_v g + C[g] g,
\end{align}
with
\begin{align}
A_{ij}[g] &= \bar a_{ij}[\mu g],\label{e:AA}\\
B_{j}[g] &=  2 \bar a_{ij}[\mu g] \frac{\partial_{v_i} \mu}{\mu} ,\label{e:BB}\\
C[g] &=  \bar c[\mu g] + \bar a_{ij}[\mu g] \frac{\partial_{v_iv_j}^2\mu}{\mu},\label{e:CC}
\end{align}
where, in general, we sum over repeated indices. Explicitly,
\begin{equation}\label{e:mu}
\begin{split}
\frac{\partial_{v_i} \mu}{\mu} &= -2(\rho_0 - \kappa t) v_i,\\
\frac{\partial_{v_iv_j}^2 \mu}{\mu} &= -2(\rho_0 - \kappa t) \delta_{ij} + 4(\rho_0-\kappa t)^2 v_i v_j.
\end{split}
\end{equation}

The main purpose of this section is to derive the estimates on the coefficients $A$, $B$, and $C$ defined in \eqref{e:AA}--\eqref{e:CC}, as well as $\bar a$ and $\bar c$ defined in \eqref{e:a}--\eqref{e:c}, that will be needed in Section \ref{s:LWP}. 

\begin{lemma}\label{l:coeffs} Let $\gamma \in [-3,0)$, let $p\in [2,\infty]$, let $\alpha$ and $\beta$ be multi-indices, let $g$ be a function on $\R^6$ such that $\partial_x^\alpha  g(x,\cdot) \in W^{|\beta|,p}_v(\R^3)$ for all $x\in \R^3$, and let $\mu(v) = e^{-\lambda |v|^2}$ for some $\lambda>0$.
	\begin{enumerate}
		\item[(a)] For any unit vector $e \in \mathbb S^2$,
		\[	|\partial_x^\alpha\partial_v^\beta (\bar a_{ij}[\mu g]e_i e_j)(x,v)| \lesssim  \|\partial_x^\alpha g(x,\cdot)\|_{W^{|\beta|,p}_v(\R^3)} \vv^{\gamma+2}.\]
		In addition, we have the following improved bounds in the $v$ direction:
		\begin{align*}
		|\partial_x^\alpha\partial_v^\beta (\bar a_{ij}[\mu g]v_i v_j)(x,v)| &\lesssim  \|\partial_x^\alpha g(x,\cdot)\|_{W^{|\beta|,p}_v(\R^3)} \vv^{\gamma+2},\\
		|\partial_x^\alpha\partial_v^\beta(\bar a_{ij}[\mu g]v_j)(x,v)| &\lesssim \|\partial_x^\alpha g(x,\cdot)\|_{W^{|\beta|,p}_v(\R^3)}\vv^{\gamma+2}.
		\end{align*}
		 \item[(b)] For $p=\infty$, one has
		\[	|\partial_x^\alpha\partial_v^\beta\bar c[\mu g](x,v)| \lesssim \| \partial_x^\alpha g(x,\cdot)\|_{W^{|\beta|,\infty}_v(\R^3)} \vv^{\gamma}. \]
		\item[(c)] If  
		$p\geq 2$ and $\theta > 3 + p \gamma$, one has
		\begin{align*}
		\int_{\R^3} \vv^{-\theta}|\partial_x^\alpha \partial_v^\beta \bar c[\mu g](x,v)|^p\dd v &\lesssim \|\partial_x^\alpha g(x,\cdot)\|_{W^{|\beta|,p}_v(\R^3)}^p.
		\end{align*}
	\end{enumerate}
	The implied constants depend only on $\alpha$, $\beta$, $\gamma$, $\lambda$, and, where appropriate, $\theta$ and $p$.
\end{lemma}
\begin{proof}

	Note that for any $\alpha$ and $\beta$, one has $\partial_x^\alpha\partial_v^\beta \bar a_{ij}[\mu g] = \bar a_{ij}[\partial_x^\alpha \partial_v^\beta (\mu g)]$ and $\partial_x^\alpha\partial_v^\beta \bar c[\mu g] = \bar c[\partial_x^\alpha \partial_v^\beta(\mu g)]$. Also note that
	\begin{equation}\label{e:dv}
	|\partial_x^\alpha \partial_v^\beta (\mu g)| \lesssim \sum_{|\beta'|\leq |\beta|}|\sqrt \mu \partial_x^\alpha \partial_v^{\beta'} g|.
	\end{equation}

	With $p'\in [1,2]$ such that $1/p+1/p'=1$, we have from \eqref{e:a} and \eqref{e:dv} that
	\begin{align*}
	|\partial_x^\alpha\partial_v^\beta \bar a_{ij}[\mu g](v)| 
	&\lesssim \sum_{|\beta'|\leq |\beta|}\int_{\R^3}|w|^{\gamma+2}|\partial_x^\alpha \partial_v^{\beta'} g(v-w)|\sqrt{ \mu(v-w)}\dd w\\
	&\lesssim \| \partial_x^\alpha g\|_{W^{|\beta|,p}_v(\R^3)}\left(\int_{\R^3}  \mu^{p'/2}(v-w)|w|^{p'(\gamma+2)}\dd w\right)^{1/p'}
	\lesssim \|\partial_x^\alpha g\|_{W^{|\beta|,p}_v(\R^3)}\vv^{\gamma+2},
	\end{align*}
	where we use that $|w|^{p'(\gamma+2)}$ is integrable near the origin since $p'(\gamma+2) > -3$.  
	
	Next, we show that the quadratic form $e\cdot(\partial_x^\alpha \partial_v^\beta \bar a e)$ has improved upper bounds in the $v$ direction. First, suppose $|\alpha|+|\beta| = 0$. Following the calculations of \cite{cameron2017landau}, we have for $w = v-z\in \R^3$,
	\begin{align*}
	v\cdot \left(I - \frac{w}{|w|}\otimes \frac{w}{|w|}\right) v\,|w|^{\gamma+2}
	&= |v|^2\left(1 - \left(\frac {(v-z)\cdot e}{|v-z|}\right)^2\right)|v-z|^{\gamma+2}\\
	&= |v|^2\left(|v-z|^2 - \left(|v|-z\cdot e\right)^2\right)|v-z|^{\gamma}\\
	&= |v|^2\left(|z|^2 - (z\cdot e)^2\right)|v-z|^{\gamma}
	= |v|^2|z|^2\sin^2\varphi |v-z|^{\gamma},
	\end{align*}
	where $\varphi$ is the angle between $v$ and $z$. Let $R = |v|/2$. If $z\in B_R(v)$, then $|\sin \varphi| \leq |v-z|/|v|$, and
	\begin{align*}
	\int_{B_R(v)} |v|^2|z|^2\sin^2\varphi& |v-z|^\gamma  \mu g(z)\dd z \leq \int_{B_R(v)} |z|^2 |v-z|^{\gamma+2} \mu g\dd z\\
	&\lesssim \vv^{\gamma+2}\left(\int_{B_R(v)} \mu^{p/2}  g^p\dd z\right)^{1/p}
	\lesssim \|g\|_{L_v^{p}(\R^3)}\vv^{\gamma+2}.
	\end{align*}
	On the other hand, if $z\notin B_R(v)$, i.e. $|v-z| \geq R=|v|/2$, then $|v-z|^\gamma \lesssim |v|^\gamma$, and we have
	\begin{align*}
	\int_{\R^d\setminus B_R(v)} |v|^2|z|^2\sin^2\varphi |v-z|^\gamma  \mu g(z) \dd z &\lesssim  \|g\|_{L_v^p(\R^3)} \vv^{\gamma+2}.
	\end{align*}
	Still under the assumption that $|\alpha|+|\beta|= 0$, we now show the improved bound on $\bar a_{ij} v_i$. Let $\{v/|v|, e_2, e_3\}$ be an orthonormal basis for $\R^3$ and write 
	\[\bar a[ \mu  g] v = \beta_1 \frac v{|v|} + \beta_2 e_2 + \beta_3 e_3.\] 
	The above bound on $v\cdot (\bar a v)$ implies $\beta_1 = (v/|v|)\cdot (\bar a[\mu g] v)\lesssim \| g\|_{L^p_v(\R^3)}\vv^{\gamma+1}$. Since $\bar a_{ij}[\mu  g]$ is positive-definite, we have 
	\[\beta_2 = e_2\cdot(\bar a [\mu g] v) \leq \frac 1 2 v\cdot (\bar a [\mu  g] v) + \frac 1 2 e_2 \cdot (\bar a [\mu g] e_2) \lesssim \| g\|_{L^p_v(\R^3)}\vv^{\gamma+2}.\] 
	Similarly, $\beta_3 \lesssim \| g\|_{L^p_v(\R^3)}\vv^{\gamma+2}$. We conclude $\bar a_{ij} v_i \lesssim \|g\|_{L_v^p(\R^3)}\vv^{\gamma+2}$.

For $|\alpha|+|\beta| >0$, we write
\begin{equation}\label{eq:12154}
	\partial_x^\alpha \partial_v^\beta (\bar a [\mu g] v)
		= \sum_{\beta' + \beta'' + \beta''' = \beta}C_{\beta',\beta'', \beta'''} \bar a [\partial_v^{\beta'}\mu \partial_v^{\beta''}\partial_x^\alpha g] \partial_v^{\beta'''}v
\end{equation}
If $|\beta'''| = 1$, then letting $i$ be the unique index such that $\beta'''_i = 1$, we see that
\begin{equation}\label{eq:12153}
	|\bar a [\partial_v^{\beta'}\mu \partial_v^{\beta''}\partial_x^\alpha g] \partial_v^{\beta'''}v|
		= |\bar a [\partial_v^{\beta'}\mu \partial_v^{\beta''}\partial_x^\alpha g] e_i|
		\lesssim \|\partial_x^\alpha g\|_{W^{|\beta''|,p}_v}\vv^{\gamma+2}
		\leq \|\partial_x^\alpha g\|_{W^{|\beta|,p}_v}\vv^{\gamma+2},
\end{equation}
where the second-to-last inequality follows exactly as above.  If $|\beta'''| = 0$, then we write $\partial_v^{\beta'} \mu \partial_x^\alpha \partial_v^{\beta''}g = \sqrt \mu (g_+ - g_-)$, where $g_+$ and $g_-$ are both non-negative. Then
\[
	\bar a [\partial_v^{\beta'}\mu \partial_v^{\beta''}\partial_x^\alpha g] \partial_v^{\beta'''}v
		= \bar a[\sqrt \mu g_+]v + \bar[\sqrt \mu g_-]v.
\]
Since $g_+$ and $g_-$ are both non-negative, we apply our work from the case $|\alpha|+|\beta|=0$ to see that 
\begin{equation}\label{eq:12152}
	|\bar a [\partial_v^{\beta'}\mu \partial_v^{\beta''}\partial_x^\alpha g] \partial_v^{\beta'''}v|
		\leq |\bar a[\sqrt \mu g_+]v| + |\bar a[\sqrt \mu g_-]v|
		\lesssim \|g_+\|_{L^p_v(\R^3)}\vv^{\gamma+2} + \|g_-\|_{L^p_v(\R^3)}\vv^{\gamma+2}.
\end{equation}
Since $|\partial_v^{\beta'} \mu| \lesssim \sqrt \mu$, it is clear that
\begin{equation}\label{eq:12151}
	\|g_+\|_{L^p_v(\R^3)}^p + \|g_-\|_{L^p_v(\R^3)}^p
		= \|\mu^{-1/2} \partial_v^{\beta'}\mu \partial_x^\alpha \partial_v^{\beta''}g\|_{L^p_v(\R^3)}^p
		\lesssim \|\partial_x^\alpha g\|_{W^{|\beta''|,p}_v(\R^3)}
		\leq \|\partial_x^\alpha g\|_{W^{|\beta|,p}_v(\R^3)}.
\end{equation}
Combining~\eqref{eq:12154},~\eqref{eq:12153},~\eqref{eq:12152}, and~\eqref{eq:12151}, we obtain the desired estimate:
\[
	|\partial_x^\alpha \partial_v^\beta (\bar a[\mu g] v)|
		\lesssim \|\partial_x^\alpha g\|_{W^{|\beta|,p}_v(\R^3)}\vv^{\gamma+2}.
\]
Proceeding in a similar manner and using our bound on $\partial_x^\alpha\partial_v^\beta (\bar a[\mu g] v)$, we can show 
\[
	|\partial_x^\alpha \partial_v^\beta (v\cdot \bar a[\mu g] v)|
		\lesssim \|\partial_x^\alpha g\|_{W^{|\beta|,p}_v(\R^3)}\vv^{\gamma+2}.
\]
This establishes (a).

	
For (b), if $\gamma \in (-3,0)$, \eqref{e:c} implies
	\begin{align*}
	|\partial_x^\alpha\partial_v^\beta \bar c[\mu g](v)| &\lesssim \sum_{|\beta'|\leq |\beta|}\int_{\R^3}|w|^{\gamma}|\partial_x^\alpha \partial_v^{\beta'} g(v-w)|\sqrt \mu(v-w)\dd w\\
	&\lesssim \| \partial_x^\alpha g\|_{W^{|\beta|,\infty}_v(\R^3)}\left(\int_{\R^3} \sqrt \mu(v-w)|w|^{\gamma}\dd w\right)
	\lesssim \| \partial_x^\alpha g\|_{W^{|\beta|,\infty}_v(\R^3)}\vv^{\gamma},
	\end{align*}
	since $\gamma  > -3$. If $\gamma = -3$, then $\bar c[\mu g] = c_{3} \mu g$, and an even stronger bound is satisfied.
	
	To prove (c), in the case $\gamma = -3$, the desired estimate is an immediate consequence of the formula $\bar c[\mu g] = c_3\mu g$. Letting $\gamma \in (-3,0)$, we restrict to the case $|\alpha|+|\beta|=0$ for brevity; the remaining cases follow easily from \eqref{e:dv}.  Using H\"older's inequality,
	\begin{align*}
	| \bar c[\mu g]|^p &\lesssim \int_{\R^3} g(v-w)^p |w|^{\gamma} \mu(v-w) \dd w \left( \int_{\R^3}  \mu(v-w) |w|^{\gamma}\dd w\right)^{p/p'}\\
	&\lesssim  \int_{\R^3} \vv^{p\gamma/p'} |w|^{\gamma}|g(v-w)|^p \mu(v-w) \dd w,
	\end{align*}
	where $p'$ is the dual exponent to $p$. Note that $p/p' = p-1$. This implies that 
	\begin{align*}
	\int_{\R^3} \vv^{-\theta}& | \bar c[\mu g]|^p\dd v
	\lesssim \int_{\R^3}\int_{\R^3} \vv^{-\theta + \gamma(p-1)} |w|^{\gamma} |g(v-w)|^p \mu(v-w)\dd w\dd v\\
	&\lesssim \int_{\R^3} \int_{\R^3} \langle w\rangle^{- \theta + \gamma(p-1)} |w|^{\gamma}|g(v-w)|^p \dd w \dd v
	\lesssim \|g\|_{L^p_v}^p \int_{\R^3} \langle w\rangle^{- \theta + \gamma(p-1)} |w|^{\gamma} \dd w,
	\end{align*}
	by Fubini's Theorem and the estimate $\|\vv^r \mu(v-w)\|_{L^\infty_v(\R^3)}\lesssim \langle w\rangle^r$ for any $r\in\R$. The last integral is finite because, by assumption, $-\theta + \gamma p < -3$. 	
\end{proof}

\begin{lemma}\label{l:abc} Let $\gamma \in [-3,0)$, $\mu = e^{-\lambda\vv^2}$ for some $\lambda>0$, and let $g$, $\alpha$, and $\beta$ be such that $\partial_x^\alpha g(x,\cdot)\in H^{|\beta|}_v$ for $x\in \R^3$. Let $A[g]$, $B[g]$, and $C[g]$ be defined by \eqref{e:AA}--\eqref{e:CC}. Then we have
		\begin{align*}
		|\partial_x^\alpha \partial_v^\beta A[g](x,v)| &\lesssim \|\partial_x^\alpha g(x,\cdot)\|_{H^{|\beta|}_v} \vv^{\gamma+2},\\
		|\partial_x^\alpha \partial_v^\beta B[g](x,v)| &\lesssim \|\partial_x^\alpha g(x,\cdot)\|_{H^{|\beta|}_v} \vv^{\gamma+2}.
		\end{align*}
		If $\partial_x^\alpha g(x,\cdot) \in W^{|\beta|,\infty}_v$ for $x\in \R^3$, then
		\begin{align*}
		|\partial_x^\alpha \partial_v^\beta C[g](x,v)| &\lesssim \|\partial_x^\alpha g(x,\cdot)\|_{W^{|\beta|,\infty}_v} \vv^{\gamma+2}.
		\end{align*}
	The implied constants depend only on $\alpha$, $\beta$, $\gamma$, and $\lambda$.
\end{lemma}
\begin{proof}
	The bounds on $A[g]$ and $B[g]$ follow immediately from Lemma \ref{l:coeffs}(a) with $p=2$, and the bound on $C[g]$ follows from Lemma \ref{l:coeffs}(a) and (b) with $p=\infty$.	
\end{proof}

\section{Local existence}\label{s:LWP}

In this section, we solve \eqref{e:g} on a time interval $[0,T]$. To do this, we first consider a linearization of \eqref{e:g} with added viscosity on a bounded domain.  Let us introduce the following notation: for any $\eps>0$ and $R>3$, define the mollifier $\zeta_\eps(x,v) = \eps^{-6} \zeta(x/\eps, v/\eps)$ for some non-negative, $C_c^\infty$ function $\zeta$ such that $\int \zeta \dd x \dd v = 1$.  Next, let $\Omega_R = \{(x,v) \in \R^6: |x|^2 + |v|^2 < R^2\}$ be a ball in $\R^6$ centered at the origin.  
Finally, let $\chi_R$ be a smooth cutoff function on $\R^6$, supported in $\Omega_{R-1}$, equal to $1$ in $\Omega_{R-2}$, radially symmetric, 
monotone, and such that $|D_{x,v}^n \chi_R| \leq 2^n$ for any $n\in \N\cup\{0\}$. 

\begin{lemma}\label{l:parabolic}
 Let $g_{in} \in H^k_{\ul}$ and $h \in L^\infty([0,T], H^k_{\ul})$ be given nonnegative
functions with $T>0$.  
For any $\eps>0$, let $h_\eps = \zeta_\eps * h$. 
Then, for all $R$ sufficiently large, there exists a unique solution $G = G_{h,R,\eps}$ to
\begin{equation}\label{e:glin}
\partial_t G = \eps \Delta_{x,v} G + \textup{tr} \left( A[h_\eps] D_v^2 G \right) + B[h_\eps] \cdot \nabla_v G
- v \cdot \nabla_x G + \left( C[h_\eps] - \kappa \vv^2 \right) G
\end{equation}
on $[0,T] \times \Omega_R$ with initial data and boundary values
\begin{equation}\label{e:ginit}
G(0,x,v) = \chi_R(x,v) \left( \zeta_\eps * g_{in} \right)(x,v) \ \text{ and } \
G(t,y,w) = 0 \ \text{ for all } (t,y,w) \in [0,\infty) \times \partial \Omega_R. 
\end{equation}
The solution $G$ is nonnegative and $G\in C^\infty([0,T]\times\Omega_R)$.
\end{lemma}

Lemma \ref{l:parabolic} follows from standard parabolic theory. In particular, the nonnegativity of $G$ is
implied by the maximum principle. As such, we omit the proof. 

Our next step is to solve
the linearized problem \eqref{e:glin} on the whole space and with $\eps = 0$. We do this by
looking at the solutions $G_{h,R,\eps}$ of \eqref{e:glin} above and extracting a weak limit as $R$ tends to $\infty$ and $\eps$ tends to zero.

\begin{lemma}\label{l:gre}
Let $T \in (0,T_{\rho_0,\kappa}]$, let $h \in  L^\infty([0,T], H^k_{\ul})$ and $g_{in} \in H^k_{\ul}$ be given nonnegative functions.
Then there exists a solution $G \in Y^k_{T}$ to the linearized problem
\begin{equation}\label{e:true-glin}
\partial_t G + v \cdot \nabla_x G + \kappa \vv^2 G =
\textup{tr} \left( A[h] D_v^2 G \right) + B[h] \cdot \nabla_v G + C[h] G
\end{equation}
with $G(0,x,v) = g_{in}(x,v)$. Moreover, $G$ is nonnegative and
\begin{equation}\label{e:main-bound}
\| G \|_{Y^k_{T}}^2 \leq \| g_{in} \|^2_{H^k_{\ul}}
\exp \left( C_1 T \left( 1 + \| h \|_{L^\infty([0,T], H^k_{\ul})}^\Lambda \right) \right)
\end{equation}
for $\Lambda= \max\{8,2/|\gamma| +1\}$ and some $C_1>0$ depending on $k$, $\gamma$, $\rho_0$, and $\kappa$.
\end{lemma}

\begin{proof}
Before beginning, we set some notation and make some useful observations.  Let $\psi$ be a fixed smooth cut-off
function in the velocity variable; that is, $\psi$ is radial, nonnegative, identically $1$ for $|v|\leq 1$,
vanishes for $|v| \geq 11/10$, and monotonic. For $0<r<R$, define $\psi_r(v) = \psi(v/r)$.
We then define
\begin{equation}\label{e:locnorm}
\begin{split}
&\| g \|_{{\dot H}^{k,l}_{\ul,r}(\Omega_R)} = \sum_{|\alpha|+|\beta|=k}
\sup_{a \in B_{R/10}} \int_{\Omega_R} |\phi(x-a)\psi_r(v) \vv^l
\partial_x^\alpha\partial_v^\beta g(x,v)|^2 \dd x \dd v, \\
 &\| g \|_{\dot{Y}^k_{T,r,R}} = \| g \|_{L^\infty([0,T], \dot{H}^{k,0}_{\ul,r}(\Omega_R))} +
\| g \|_{L^2([0,T], \dot{H}^{k,1}_{\ul,r}(\Omega_R))}, ~ \text{ and } 
~ \| g \|_{Y^k_{T,R}} = \sum_{m=0}^k \| g \|_{\dot{Y}^m_{T,R/2^m,R}}.
\end{split}
\end{equation}
We note that the higher-derivative norms in $\| \cdot \|_{Y^k_{T,R}}$ are more strongly localized.
In particular, notice that $Y^0_{T,R}$ coincides with $L^\infty([0,T], L^2_{\ul}(\Omega_R))
\cap L^2([0,T], H^{0,1}_{\ul}(\Omega_R))$ since $\psi_R \equiv 1$ on $\Omega_R$.  We also mention that the support of $\psi$ and the admissible $a$ in the supremum are chosen so that, when $m \geq 1$, $\supp(\phi(\cdot-a)\psi_{R/2^m})$ avoids the boundary of $\Omega_R$ when $R \geq 40 \sqrt 3 / 7$.

We write $\phi = \phi(x-a)$, $\phi_2 = \phi((x-a)/2)$, and $\psi = \psi(v)$. We frequently use the following facts: 
\begin{enumerate}
\item[(i)] $\phi = \phi \phi_2$;
\item[(ii)] $|\nabla_x \phi| \lesssim \phi_2$;
\item[(iii)] for any $s, \eta > 0$, $\vv^{2-s} \leq \eta \vv^2 + C_s \eta^{-\max\{2/s - 1,0\}}$ for $C_s>0$ depending only on $s$;
\item[(iv)] replacing $\phi$ with $\phi_2$ in the definition \eqref{e:ul-def} of $\|\cdot\|_{H^k_{\ul}}$ defines an equivalent norm.
\end{enumerate}
  Also, recall that if $\partial_x^\alpha \partial_v^\beta f \in L^p$ for some multi-indices $\alpha$ and $\beta$, then $\|\partial_x^\alpha \partial_v^\beta (\zeta_\eps * f) \|_{L^p} \leq \|\partial_x^\alpha \partial_v^\beta f \|_{L^p}$. This implies $A[h_\eps]$, $B[h_\eps]$, and $C[h_\eps]$ satisfy the same bounds as $A[h]$, $B[h]$, and $C[h]$ (cf.~Lemma \ref{l:abc}), with constants independent of $\eps$.

Throughout the proof, we take $N$ to be a fixed positive integer that will eventually be chosen large enough (independently of $R$ or $\eps$) that our inequalities
close correctly. Also, we denote by $C_1$ a running constant that is independent of $R$ and $\eps$
(but may depend on $N$).  Finally, we denote by $\tilde C$ a running constant independent of $R$, $\eps$, and $N$.

For given $R$ and $\eps$, let $G_R = G_{h,R,\eps}$ be the solution to \eqref{e:glin} with boundary conditions \eqref{e:ginit} on
$\Omega_R$, guaranteed by Lemma \ref{l:parabolic}. We will establish a bound on $G_R$ in $Y^k_{T,R}$
(uniform in $R$ and $\eps$) that will then allow us to take a limit as $R$ tends to $\infty$ and $\eps$
tends to zero.

For ease of notation, we define
\begin{equation}\label{def:XY}
\begin{split}
 &X = X(t):= \|h \|_{L^\infty([0,t],H^k_{\ul})}, \qquad
 Y_{m,R} = Y_{m,R}(t):= \| G_R \|_{L^\infty([0,t], \dot{H}^{m,0}_{\ul,R/2^m}(\Omega_R))},\\
 &\text{ and }\qquad
 Z_{k,R} = Z_{k,R}(t):= \left( \sum_{m=0}^k Y_{m,R}^2 \right)^{\frac 1 2}. \\
\end{split}
\end{equation}

We now begin the mechanics of the proof.  Our goal is to prove that
\begin{equation}\label{e:base-bound}
\frac d{dt} \| G_R \|_{\dot{H}^{0,0}_{\ul,R}(\Omega_R)}^2 +
\kappa \| G_R \|_{\dot{H}^{0,1}_{\ul,R}(\Omega_R)}^2
+ \eps \| G_R \|_{\dot{H}^{1,0}_{\ul,R}(\Omega_R)}^2
\leq C_1 \left( 1 + X^\Lambda \right) Z_{0,R}^2,
\end{equation}
and, by induction, that
\begin{equation}\label{e:induction-bound}
\frac d{dt} \| G_R \|^2_{\dot{H}^{m,0}_{\ul,R/2^m}(\Omega_R)}
+ \kappa \| G_R \|^2_{\dot{H}^{m,1}_{\ul,R/2^m}(\Omega_R)}
+ \eps \| G_R \|^2_{\dot{H}^{m+1,0}_{\ul,R/2^m}(\Omega_R)}
\leq C_1 \left( 1 + X^\Lambda \right) Z_{m,R}^2
+ \tilde{G}_m,
\end{equation}
for all $1 \leq m \leq k$ and all $R$ sufficiently large (depending on the parameters and the data).
Recall that $C_1$ is a fixed constant that depends on $k$, $\rho_0$, $\kappa$, $\gamma$, and $N$, but not on $R$ or $\eps$.
Here $\tilde{G}_m$ is a time-integrable function such that, if $R$ is sufficiently large in a way we make explicit in the sequel,
\[ \int_0^T \tilde{G}_m(t) dt \leq \eps. \]

\subsubsection*{Base case:}
Recall that, for $m=0$, there is no longer any cut-off in $v$ in the $\dot{H}$-norms.
We multiply \eqref{e:glin} by $\phi^2 G_R$ and integrate over $\Omega_R$. Since $G_R = 0$ on
$\partial \Omega_R$, we can integrate by parts without boundary terms, yielding
\begin{align*}
\frac 1 2 \frac d{dt} &\| G_R \|_{L^2(\Omega_R)}^2 + \kappa \int_{\Omega_R} \phi^2 \vv^2 G_R^2 \dd x \dd v
+\eps \int_{\Omega_R} \phi^2 |\nabla G_R|^2 \dd x \dd v\\
&\leq \int_{\Omega_R} \phi v \cdot \nabla_x \phi G_R^2 \dd x \dd v
 -\int_{\Omega_R} \phi^2 \nabla_v G_R \cdot A[h] \cdot \nabla_v G_R \dd x \dd v\\
&\quad-\int_{\Omega_R} \phi^2 G_R \left( \nabla_v \cdot A[h_\eps] \right) \cdot \nabla_v G_R \dd x \dd v
- \frac 1 2 \int_{\Omega_R} \phi^2 G_R^2 \nabla_v \cdot B[h_\eps] \dd x \dd v
+ \int_{\Omega_R} \phi^2 C[h_\eps] G_R^2 \dd x \dd v\\
&=: I_1 + I_2 + I_3 + I_4 + I_5.
\end{align*}
Since $h$ is nonnegative, so is $h_\eps$, and therefore $A[h_\eps]$ is nonnegative definite. We may then ignore $I_2$,
since it is nonpositive. The transport term is easily bounded using Young's inequality:
\begin{equation*}
 |I_1| \leq
\frac \kappa N \int_{\Omega_R} \phi \phi_2 \vv^2 G_R^2 \dd x \dd v
+ C_1 \int_{\Omega_R} \phi \phi_2 G_R^2 \dd x \dd v
\leq \frac \kappa N \| G_R \|_{\dot{H}^{0,1}_{\ul,R}}^2 + C_1 Z_{0,R}^2.
\end{equation*}
For $I_3$, we note that $G_R$ vanishes on the boundary of $\Omega_R$, allowing us to integrate by parts without
boundary terms. We then use Lemma \ref{l:abc} and Young's inequality with $\eta > 0$ to obtain
\begin{align*}
 I_3 &= - \frac 1 2 \int_{\Omega_R} \phi^2 (\nabla_v \cdot A[h_\eps]) \cdot \nabla_v (G_R^2) \dd x \dd v
\lesssim \frac 1 2 \int_{\Omega_R} \phi^2 |D_v^2 A[h_\eps]| G_R^2 \dd x \dd v \\
&\lesssim \int_{\Omega_R} \phi^2 \| h_\eps \|_{H^2_{v}} \vv^{2+\gamma} G_R^2 \dd x \dd v
\leq \eta X \| \phi \vv G_R \|_{L^2(\Omega_R)}^2 + \eta^{-\max\{2/\delta-1,0\}} X \| \phi G_R \|_{L^2(\Omega_R)}^2.
\end{align*}
Setting $\eta = \kappa N^{-1} X^{-1}$ yields $I_3 \leq \frac \kappa N \| G_R \|_{\dot{H}^{0,1}_{\ul,R}(\Omega_R)}^2 + C_1 X^{\Lambda} Z_{0,R}^2.$
The remaining terms are bounded in a similar way. That is, using Lemma \ref{l:abc}, 
\begin{align*}
|I_4|  &\lesssim \int_{\Omega_R} \phi^2 G_R^2
|\nabla_v B[h_\eps]| \dd x \dd v \\
&\lesssim \int_{\Omega_R} \phi^2 \| h_\eps \|_{H^1_{v}} \vv^{2+\gamma} G_R^2 \dd x \dd v
\leq \frac \kappa N \| G_R \|_{\dot{H}^{0,1}_{\ul,R}(\Omega_R)}^2 + C_1 X^{\Lambda} Z_{0,R}^2,
\end{align*}
and using Lemma \ref{l:abc} and the Sobolev embedding, i.e.~$H^k(\R^6)\subset L^\infty(\R^6)$,
\begin{align*}
 |I_5| &\lesssim \int_{\Omega_R} \phi^2 G_R^2
|C[h_\eps]|\dd x \dd v
\lesssim \int_{\Omega_R} \phi^2 \| h_\eps \|_{L^\infty_{v}} \vv^{2+\gamma} G_R^2 \dd x \dd v\\
&\leq \frac \kappa N \| G_R \|_{\dot{H}^{0,1}_{\ul,R}(\Omega_R)}^2 + C_1 X^{\Lambda} Z_{0,R}^2.
\end{align*}
Summing over all $\alpha$ and $\beta$, taking a supremum in $a$, and
choosing $N$ large enough yields \eqref{e:base-bound}.

\subsubsection*{Induction Step:}
Let $\alpha$ and $\beta$ be multi-indices with $|\alpha|+|\beta| = m \leq k$. Let $\psi_m = \psi_{R/2^m}$.
Applying $\partial_x^\alpha \partial_v^\beta$ to \eqref{e:glin}, multiplying by
$\phi^2 \psi_m^2 \partial_x^\alpha \partial_v^\beta G_R$, and integrating over $\Omega_R$ yields
\begin{equation}\label{e:main-est}
\begin{split}
\frac{1}{2}\frac d {dt}&\int_{\Omega_R} \phi^2 \psi_m^2 |\partial_x^\alpha \partial_v^\beta G_R|^2 \dd x \dd v +
\kappa\int_{\Omega_R} \phi^2 \psi_m^2 \langle v \rangle^2 |\partial_x^\alpha \partial_v^\beta G_R|^2 \dd x \dd v\\
&= J_1 + J_2 + J_3 + J_4 + J_5 + J_6,
\end{split}
\end{equation}
where
\begin{equation}\label{e:J1J5}
\begin{split}
J_1 &= \int_{\Omega_R} \phi^2 \psi_m^2 \partial_x^\alpha\partial_v^\beta \text{tr}\left(A[h_\eps]
D_v^2 G_R\right) \partial_x^\alpha \partial_v^\beta G_R \dd x \dd v,\\
J_2 &= \int_{\Omega_R} \phi^2 \psi_m^2 \partial_x^\alpha \partial_v^\beta \left( B[h_\eps]
\cdot \nabla_v G_R \right) \partial_x^\alpha \partial_v^\beta G_R \dd x \dd v,\\
J_3 &= \int_{\Omega_R} \phi^2 \psi_m^2 \partial_x^\alpha \partial_v^\beta(C[h_\eps] G_R)
\partial_x^\alpha\partial_v^\beta G_R \dd x \dd v,\\
J_4 &= -\int_{\Omega_R} \phi^2 \psi_m^2 \left[ \partial_x^\alpha \partial_v^\beta, \
v \cdot \nabla_x + \kappa \vv^2 \right] G_R \partial_x^\alpha \partial_v^\beta G_R \dd x \dd v = \\
&= - \sum_{j=1}^3 \int_{\Omega_R}\phi^2 \psi_m^2 \left(\partial_x^{\alpha+e_j} \partial_v^{\beta-e_j} +
2\kappa(v_j \partial_x^\alpha\partial_v^{\beta-e_j} + \partial_x^\alpha\partial_v^{\beta-2e_j} )\right)
G_R \partial_x^\alpha\partial_v^\beta G_R \dd x \dd v,\\
J_5 &=  - \int_{\Omega_R} \phi^2 \psi_m^2  v\cdot \nabla_x \partial_x^\alpha \partial_v^\beta G_R
\partial_x^\alpha \partial_v^\beta G_R \dd x \dd v,\quad\text{ and}\\
J_6 &= \eps \int_{\Omega_R} \phi^2 \psi_m^2 \Delta \partial_x^\alpha\partial_v^\beta G_R
\partial_x^\alpha \partial_v^\beta G_R \dd x \dd v.
\end{split}
\end{equation}
Here, $\partial_x^{\alpha+e_1}$ means $\partial_{x_1}^{\alpha_1+1}\partial_{x_2}^{\alpha_2}
\partial_{x_3}^{\alpha_3}$, etc., $\partial_{v_j}^{i} G_R = 0$ if $i<0$, and the $[\cdot,\cdot]$ in $J_4$ is the commutator.
Note that we have not yet integrated by parts in any variable.
We proceed to bound each of the six terms above.

\subsubsection*{Estimating $J_1$:}
By the product rule we have
\begin{align*}
J_1 &= \sum_{\substack{\alpha'+\alpha''=\alpha \\ \beta'+\beta''=\beta}} \int_{\Omega_R} \phi^2 \psi_m^2
\text{tr} \left( \partial_x^{\alpha'} \partial_v^{\beta'}(A[h_\eps])
D_v^2 \partial_x^{\alpha''} \partial_v^{\beta''} G_R \right)
\partial_x^\alpha \partial_v^\beta G_R \dd x \dd v.
\end{align*}
We must use different techniques depending on the distribution of derivatives.

\subsubsection*{$J_1$ terms with $|\alpha'|+|\beta'| = m \geq 2$:}
In this case, $\alpha'=\alpha$ and $\beta'=\beta$.
Lemma \ref{l:abc} yields
\begin{equation*}
\begin{split}
\int_{\Omega_R} &\phi^2 \psi_m^2 |\partial_x^{\alpha}\partial_v^{\beta}A[h_\eps]|
|D^2_v G_R| |\partial_x^\alpha\partial_v^\beta G_R| \dd x \dd v \\
&\lesssim \int_{\Omega_R} \phi^2 \psi_m^2 \vv^{2+\gamma} \| \partial_x^\alpha h_\eps \|_{H^{|\beta|}_v}
 |\partial_x^\alpha \partial_v^\beta G_R|| D^2_v G_R | \dd x \dd v\\
&\lesssim \int_{\Omega_R}\phi_2 \| \partial_x^\alpha h_\eps \|_{H^{|\beta|}_v}
\phi \| \psi_m \vv^{1+\gamma} \partial_x^\alpha \partial_v^\beta G_R \|_{L^2_v}
\phi \| \psi_m \vv D^2_v G_R \|_{L^2_v} \dd x.
\end{split}
\end{equation*}
By H\"{o}lder's inequality in $x$ and the Sobolev embedding, i.e.~$H^2(\R^3)\subset L^\infty(\R^3)$, we have
\begin{equation*}
\begin{split}
\int_{\Omega_R}\phi_2 & \| \partial_x^\alpha  h_\eps \|_{H^{|\beta|}_v}
\phi \| \psi_m \vv^{1+\gamma} \partial_x^\alpha \partial_v^\beta G_R \|_{L^2_v}
\phi \| \psi_m \vv D^2_v G_R \|_{L^2_v} \dd x \\
&\leq
\begin{cases}
\| \phi_2 \partial_x^\alpha h_\eps \|_{L^2_x H_v^{|\beta|}}
\| \phi\psi_m \vv^{1+\gamma} \partial_x^\alpha \partial_v^\beta G_R \|_{L^2}
\|\phi\psi_m \vv D^2_v G_R \|_{L^\infty_x L^2_v} & \text{ if } \ m \geq 4 \\
\| \phi_2 \partial_x^\alpha h_\eps \|_{L^4_x H^{|\beta|}_v}
\| \phi\psi_m \vv^{1+\gamma} \partial_x^\alpha \partial_v^\beta G_R \|_{L^2}
\|\phi\psi_m \vv D^2_v G_R \|_{L^4_x L^2_v} & \text{ if } \ m = 3 \\
\| \phi_2 \partial_x^\alpha h_\eps \|_{L^\infty_x H^{|\beta|}_v}
\| \phi\psi_m \vv^{1+\gamma} \partial_x^\alpha \partial_v^\beta G_R \|_{L^2}
\|\phi\psi_m \vv D^2_v G_R \|_{L^2} & \text{ if } \ m = 2
\end{cases} \\
&\leq X \| \phi\psi_m \vv^{1+\gamma} \partial_x^\alpha \partial_v^\beta G_R \|_{L^2}
\| G_R \|_{\dot{H}^{m,1}_{\ul,R/2^m}(\Omega_R)}.
\end{split}
\end{equation*}
Therefore,
\begin{equation*}
\begin{split}
\int_{\Omega_R} &\phi^2 \psi_m^2 |\partial_x^{\alpha}\partial_v^{\beta}A[h_\eps]|
|D^2_v G_R| |\partial_x^\alpha\partial_v^\beta G_R| \dd x \dd v \\
	& \leq \eta X \| \phi\psi_m \vv \partial_x^\alpha \partial_v^\beta G_R \|_{L^2}
\| G_R \|_{\dot{H}^{m,1}_{\ul,R/2^m}(\Omega_R)}\\
	&\qquad  + C \eta^{-\max\{\frac{1}{|\gamma|} - 1,0\}} X \| \phi\psi_m \partial_x^\alpha \partial_v^\beta G_R \|_{L^2}
\| G_R \|_{\dot{H}^{m,1}_{\ul,R/2^m}(\Omega_R)} \\
	&\leq 2 \eta X \| G_R \|_{\dot{H}^{m,1}_{\ul,R/2^m}(\Omega_R)}^2
+ C \eta^{-2\max\{1/|\gamma|-1,0\}-1} X \| G_R \|_{\dot{H}^{m,0}_{\ul,R/2^m}(\Omega_R)}^2.
\end{split}
\end{equation*}
Setting $ \eta = \kappa N^{-1} X^{-1}$ then gives
\begin{equation}\label{e:J-bound-(m)}
\begin{split}
\int_{\Omega_R} \phi |\partial_x^{\alpha}\partial_v^{\beta}A[h_\eps]|
|D^2_v G_R| |\partial_x^\alpha\partial_v^\beta G_R| \dd x \dd v
	&\leq \frac{2\kappa}{N} \|G_R\|_{\dot{H}^{m,1}_{\ul,R/2^m}(\Omega_R)}^2 + C_1 X^{\max\{2/|\gamma|,2\}} Z_{m,R}^2\\
	&\leq \frac {2\kappa} N \|  G_R \|_{\dot{H}^{m,1}_{\ul,R/2^m}(\Omega_R)}^2 + C_1 (1+X^\Lambda) Z_{m,R}^2.
\end{split}
\end{equation}
The remaining case ($|\alpha'|+|\beta'|=m=1$) is handled later, as it relies on a different approach.

\subsubsection*{$J_1$ terms with $|\alpha'|+|\beta'|=m-1 \geq 2$:} The analysis is similar to the previous case.
We use Lemma \ref{l:abc} to obtain
\begin{equation*}
\begin{split}
\int_{\Omega_R} \phi^2\psi_m^2 &|\partial_x^{\alpha'}\partial_v^{\beta'}A[h_\eps]|
|D^2_v \partial_x^{\alpha''} \partial_v^{\beta''} G_R|
|\partial_x^\alpha\partial_v^\beta G_R| \dd x \dd v\\
&\leq \int_{\Omega_R} \phi^2\psi_m^2 \vv^{2+\gamma}
\| \partial_x^{\alpha'}h_\eps \|_{H^{|\beta'|}_v} |D_v^2 \partial_x^{\alpha''}
\partial_v^{\beta''} G_R| |\partial_x^\alpha \partial_v^\beta G_R| \dd x \dd v\\
&\lesssim \int_{\Omega_R} \phi_2 \| \partial_x^{\alpha'}h_\eps \|_{H^{|\beta'|}_v}
\phi \| \psi_m \vv D_v^2 \partial_x^{\alpha''}\partial_v^{\beta''} G_R \|_{L^2_v}
\phi \| \psi_m \vv^{1+\gamma} \partial_x^\alpha \partial_v^\beta G_R \|_{L^2_v} \dd x.
\end{split}
\end{equation*}
By H\"{o}lder's inequality in $x$ and  Sobolev embedding, we have
\begin{equation}\label{e:J-bound-(m-1)}
\begin{split}
&\int_{\Omega_R} \phi_2 \| \partial_x^{\alpha'} h_\eps \|_{H^{|\beta'|}_v}
\phi \| \psi_m \vv D_v^2 \partial_x^{\alpha''}\partial_v^{\beta''} G_R \|_{L^2_v}
\phi \| \psi_m \vv^{1+\gamma} \partial_x^\alpha \partial_v^\beta G_R \|_{L^2_v} \dd x\\
&\quad\leq
\begin{cases}
\|\phi_2 \partial_x^{\alpha'} h_\eps \|_{L^4_x H^{|\beta'|}_v}
\| \phi\psi_m \vv D^2_v \partial_x^{\alpha''}\partial_v^{\beta''} G_R \|_{L^4_x L^2_v}
\| \phi\psi_m \vv^{1+\gamma} \partial_x^\alpha\partial_v^\beta G_R \|_{L^2} & \text{ if } \ m \geq 4\\
\|\phi_2 \partial_x^{\alpha'} h_\eps \|_{L^\infty_x H^{|\beta'|}_v}
\| \phi\psi_m \vv D^2_v \partial_x^{\alpha''}\partial_v^{\beta''} G_R \|_{L^2}
\| \phi\psi_m \vv^{1+\gamma} \partial_x^\alpha\partial_v^\beta G_R \|_{L^2} & \text{ if } \ m = 3\\
\end{cases}\\
&\quad\leq X \| \phi\psi_m \vv^{1+\gamma} \partial_x^\alpha\partial_v^\beta G_R \|_{L^2}
\| G_R \|_{\dot{H}^{m,1}_{\ul,R/2^m}(\Omega_R)} \\
&\quad\leq \frac {2\kappa} N \|  G_R \|_{\dot{H}^{m,1}_{\ul,R/2^m}(\Omega_R)}^2 + C_1 (1+X^\Lambda) Z_{m,R}^2,
\end{split}
\end{equation}
where the last inequality was obtained in the same way as in \eqref{e:J-bound-(m)}. The remaining cases
($|\alpha'|+|\beta'| = m-1 = 1$ and $|\alpha'|+|\beta'| = m-1 = 0$) are handled later, as they rely
on a different approach.

\subsubsection*{$J_1$ terms with $2 \leq |\alpha'|+|\beta'| \leq m-2$:}
This is a generic ``middle case'' where each factor in the integrand has a mild number of derivatives.
Here we use Lemma \ref{l:abc} to write
\begin{equation}\label{e:J-bound-mid}
\begin{split}
\int_{\Omega_R} \phi^2&\psi_m^2 |\partial_x^{\alpha'}\partial_v^{\beta'}A[h_\eps]|
|D^2_v \partial_x^{\alpha''} \partial_v^{\beta''} G_R|
|\partial_x^\alpha\partial_v^\beta G_R| \dd x \dd v \\
& \lesssim \int_{\Omega_R} \phi_2 \vv^{\gamma+2}
\| \partial_x^{\alpha'}h_\eps \|_{H^{|\beta'|}_v} \phi\psi_m
|D_v^2 \partial_x^{\alpha''} \partial_v{\beta''} G_R|
\phi\psi_m |\partial_x^\alpha \partial_v^\beta G_R| \dd x \dd v\\
&\leq \| \phi_2 \partial_x^{\alpha'} h_\eps \|_{L^\infty_x H^{|\beta'|}_v}
\| \phi\psi_m \vv^{1+\gamma} \partial_x^\alpha \partial_v^\beta G_R \|_{L^2}
\| G_R \|_{\dot{H}^{m,1}_{\ul,R/2^m}(\Omega_R)} \\
&\leq \frac {2\kappa} N \|  G_R \|_{\dot{H}^{m,1}_{\ul,R/2^m}(\Omega_R)}^2 + C_1 (1+X^\Lambda) Z_{m,R}^2.
\end{split}
\end{equation}
The last lines follow by the Sobolev embedding and the same analysis as in
\eqref{e:J-bound-(m)} and \eqref{e:J-bound-(m-1)}.

\subsubsection*{$J_1$ terms with $|\alpha'|+|\beta'| =1$:}
This also includes the cases where $|\alpha'|+|\beta'| = m = 1$ and
$|\alpha'|+|\beta'| = m-1 = 1$. Here we use integration by parts to write
\begin{align*}
\int_{\Omega_R} &\phi^2\psi_m^2 \tr\left(\partial_x^{\alpha'} \partial_v^{\beta'}A[h_\eps] D^2_v \partial_x^{\alpha''}\partial_v^{\beta''} G_R\right) \partial_x^\alpha \partial_v^\beta G_R \dd x \dd v\\
	&= -\int_{\Omega_R} \phi^2 \psi_m^2 \nabla_v \partial_x^{\alpha''} \partial_v^{\beta''} G_R \cdot
\partial_x^{\alpha'}\partial_v^{\beta'}A[h_\eps] \cdot
\nabla_v \partial_x^\alpha \partial_v^\beta G_R \dd x \dd v \\
	& \quad \quad -\int_{\Omega_R} \phi^2 \psi_m^2 (\nabla_v \cdot \partial_x^{\alpha'}\partial_v^{\beta'}A[h_\eps] )
\cdot \nabla_v \partial_x^{\alpha''}\partial_v^{\beta''} G_R
\partial_x^\alpha \partial_v^\beta G_R \dd x \dd v\\
	& \quad \quad - 2 \int_{\Omega_R} \phi^2 \psi_m \nabla_v \psi_m \cdot
\partial_x^{\alpha'}\partial_v^{\beta'}A[h_\eps] \cdot
\nabla_v \partial_x^{\alpha''}\partial_v^{\beta''} G_R \partial_x^\alpha \partial_v^\beta G_R \dd x \dd v
	=: J_{1,1} + J_{1,2} + J_{1,3}.
\end{align*}
The term $J_{1,2}$ is
handled in the same way as \eqref{e:J-bound-(m)}-\eqref{e:J-bound-mid}. For $J_{1,3}$, by Lemma \ref{l:abc},
the fact that $\vv \leq 2R$ on $\Omega_R$, and the fact that $|\nabla_v \psi_m| \lesssim R^{-1}\psi_{m-1}$,
we have
\begin{equation*}
\begin{split}
J_{1,3} &\lesssim \int_{\Omega_R} \phi_2 \| \partial_x^{\alpha'} h_\eps \|_{H^{|\beta'|}_v}
\vv^{2+\gamma} \left( \psi_m^2 + |\nabla_v \psi_m|^2 \right) \phi^2
|\nabla_v \partial_x^{\alpha''}\partial_v^{\beta''}G_R | |\partial_x^\alpha \partial_v^\beta G_R|\dd x \dd v\\
& \lesssim \| \phi_2 \partial_x^{\alpha'}h_\eps \|_{L^\infty_x H^{|\beta'|}_v}
\| \phi\psi_m \vv^{1+\gamma} \nabla_v \partial_x^{\alpha''}\partial_v^{\beta''} G_R \|_{L^2}
\| \phi\psi_m \vv \partial_x^\alpha\partial_v^\beta G_R \|_{L^2}\\
&\quad\quad + R^\gamma \| \phi_2 \partial_x^{\alpha'}h_\eps \|_{L^\infty_x H^{|\beta'|}_v}
\| \phi\psi_{m-1} \nabla_v \partial_x^{\alpha''}\partial_v^{\beta''} G_R \|_{L^2}
\| \phi\psi_{m-1} \partial_x^\alpha\partial_v^\beta G_R \|_{L^2}\\
&\leq \frac {2\kappa} N \| G_R \|_{\dot{H}^{m,1}_{\ul,R/2^m}(\Omega_R)}^2
+ C_1 (1+X^\Lambda) Z_{m,R}^2 + R^\gamma X \| G_R \|_{\dot{H}^{m,0}_{\ul,R/2^{m-1}}(\Omega_R)}^2.
\end{split}
\end{equation*}
The bound for $J_{1,1}$ is a combination of the two used above.
Since $|\alpha''|+|\beta''| = m-1$, we have that
$ \partial_x^\alpha \partial_v^\beta G_R = \partial_z (\partial_x^{\alpha''}\partial_v^{\beta''} G_R)$,
with $\partial_z$ a single derivative in one coordinate (either $x$ or $v$).
Let $M = (m_{ij})$ be a symmetric matrix and let $H$ be a vector field. Then, we have the following identity:
\begin{align*}
 \int_{\R^6} H & \cdot M \cdot \partial_z H \dd x \dd v =
\sum_{i,j} \int_{\R^6} H_i m_{ij} \partial_z H_j \dd x \dd v \\
&= -\sum_{i,j} \int_{\R^6} \partial_z H_i m_{ij} H_j \dd x \dd v
- \sum_{i,j} \int_{\R^6} H_i \partial_z m_{ij} H_j \dd x \dd v
= -\frac 1 2 \int_{\R^6} H \cdot \partial_z M \cdot H \dd x \dd v.
\end{align*}
Applying this identity to $J_{1,1}$, we see that
\begin{equation}\label{e:andrei1}
\begin{split}
J_{1,1} &\leq \frac 1 2 \int_{\Omega_R} \phi^2 \psi_m^2 |\partial_z\partial_x^{\alpha'}\partial_v^{\beta'}A[h_\eps]|
|\nabla_v \partial_x^{\alpha''}\partial_v^{\beta''} G_R|^2 \dd x \dd v \\
 &\qquad+\int_{\Omega_R} | \partial_z (\phi^2 \psi_m^2) | |\partial_x^{\alpha'}\partial_v^{\beta'}A[h_\eps]|
|\nabla_v \partial_x^{\alpha''}\partial_v^{\beta''} G_R|^2 \dd x \dd v.
\end{split}
\end{equation}
The first term on the right-hand-side of \eqref{e:andrei1} is bounded in the same way as $J_{1,2}$
(see \eqref{e:J-bound-(m)}-\eqref{e:J-bound-mid}). The second term changes slightly based on
the nature of $\partial_z$. If $\partial_z$ is a derivative in $v$, then the second term is bounded in
the same way as $J_{1,3}$. If $\partial_z$ is a derivative in $x$, then
\begin{equation*}
\begin{split}
\int_{\Omega_R} | \partial_z (\phi^2 \psi_m^2) |& |\partial_x^{\alpha'}\partial_v^{\beta'}A[h_\eps]|
|\nabla_v \partial_x^{\alpha''}\partial_v^{\beta''} G_R|^2 \dd x \dd v\\
&\lesssim \int_{\Omega_R} \phi_2 \phi \psi_m^2 |\partial_x^{\alpha'}\partial_v^{\beta'} A[h_\eps]|
|\nabla_v \partial_x^{\alpha''}\partial_v^{\beta''} G_R|^2 \dd x \dd v\\
&\lesssim X \| \phi_2\psi_m \vv^{1+\frac \gamma 2}
\nabla_v\partial_x^{\alpha''}\partial_v^{\beta''} G_R \|_{L^2}^2
\lesssim \frac \kappa N \| G_R \|_{\dot{H}^{m,1}_{\ul,R/2^m}(\Omega_R)}^2 + C_1 (1+X^\Lambda) Z_{m,R}^2,
\end{split}
\end{equation*}
using Young's inequality as before. Combining the different estimates above, we see that
\begin{equation}\label{e:J-bound-(1)}
\begin{split}
\int_{\Omega_R} \phi^2\psi_m^2 & \text{tr} \left( (\partial_x^{\alpha'}\partial_v^{\beta'}A[h_\eps])
D^2_v \partial_x^{\alpha''}\partial_v^{\beta''}G_R \right)\partial_x^\alpha\partial_v^\beta G_R \dd x \dd v\\
&\leq \tilde C\frac {\kappa} N \| G_R \|_{\dot{H}^{m,1}_{\ul,R/2^m}(\Omega_R)}^2 + C_1 (1+X^\Lambda) Z_{m,R}^2
+C_1R^\gamma X \| G_R \|_{\dot{H}^{m,0}_{\ul,R/2^{m-1}}(\Omega_R)}^2.
\end{split}
\end{equation}

\subsubsection*{$J_1$ terms with $|\alpha'| = |\beta'| = 0$:}
This also includes the case where $|\alpha'|+|\beta'| = m-1 = 0$. Integrating by parts yields
\begin{align*}
\int_{\Omega_R} &\phi^2\psi_m^2 \tr\left(A[h_\eps] D^2_v \partial_x^{\alpha}\partial_v^{\beta} G_R\right) \partial_x^\alpha \partial_v^\beta G_R \dd x \dd v\\ &= -\int_{\Omega_R} \phi^2 \psi_m^2 \nabla_v \partial_x^\alpha \partial_v^\beta G_R \cdot
A[h_\eps] \cdot \nabla_v \partial_x^\alpha \partial_v^\beta G_R \dd x \dd v\\
& \quad \quad - \int_{\Omega_R} \phi^2 \psi_m^2 (\nabla_v \cdot A[h_\eps] ) \cdot
\nabla_v \partial_x^\alpha\partial_v^\beta G_R
\partial_x^\alpha \partial_v^\beta G_R \dd x \dd v\\
& \quad \quad - 2 \int_{\Omega_R} \phi^2 \psi_m \nabla_v \psi_m \cdot A[h_\eps] \cdot
\nabla_v \partial_x^\alpha\partial_v^\beta G_R \partial_x^\alpha \partial_v^\beta G_R \dd x \dd v.
\end{align*}
Crucially, the first term is nonpositive, so we may ignore it. For the rest, we integrate by parts once
more and use Lemma \ref{l:abc} and Young's inequality to obtain
\begin{equation}\label{e:J-bound-(0)}
\begin{split}
\int_{\Omega_R} &\phi^2\psi_m^2 \tr\left(A[h_\eps] D^2_v \partial_x^{\alpha}\partial_v^{\beta} G_R\right) \partial_x^\alpha \partial_v^\beta G_R \dd x \dd v\\
&\lesssim \int_{\Omega_R} \phi^2 |\partial_x^\alpha\partial_v^\beta G_R|^2
\left( \psi_m^2 |D_v^2 A[ h_\eps] + \psi_m |\nabla_v \psi_m| |\nabla_v A[ h_\eps]| +
(|\nabla_v \psi_m|^2+\psi_m |D_v^2 \psi_m|) |A[h_\eps]| \right) \dd x \dd v\\
&\lesssim \int_{\Omega_R} \phi^2 |\partial_x^\alpha\partial_v^\beta G_R|^2 \vv^{2+\gamma}
\phi_2 \| h_\eps \|_{L^\infty_x H^2_v} \left( \psi_m^2 + R^{-2}\psi_{m-1}^2 \right)
\dd x \dd v\\
&\leq \frac {2\kappa} N \| G_R \|_{\dot{H}^{m,1}_{\ul,R/2^m}(\Omega_R)}^2
+ C_1 (1+X^\Lambda) Z_{m,R}^2 + C_1 R^\gamma \| G_R \|_{\dot{H}^{m,0}_{\ul,R/2^{m-1}}(\Omega_R)}^2,
\end{split}
\end{equation}
analogously to the above estimate for $J_{1,3}$.

Combining \eqref{e:J-bound-(m)}-\eqref{e:J-bound-mid}, \eqref{e:J-bound-(1)}, and \eqref{e:J-bound-(0)},
we see that
\begin{equation}\label{e:J-1}
J_1 \leq \tilde{C} \frac \kappa N \| G_R \|_{\dot{H}^{m,1}_{\ul,R/2^m}(\Omega_R)}^2
+ C_1 (1+X^\Lambda) Z_{m,R}^2 + C_1 R^\gamma X \| G_R \|_{\dot{H}^{m,0}_{\ul,R/2^{m-1}}(\Omega_R)}^2.
\end{equation}

\subsubsection*{Estimating $J_2$:}
Next we consider the integral term $J_2$, which can be written as
\[
J_2 = \sum_{\substack{\alpha'+\alpha''=\alpha\\ \beta'+\beta'' = \beta}}C_{\alpha,\beta,\alpha',\beta'}\int_{\Omega_R}
\phi^2\psi_m^2 \partial_x^{\alpha'}\partial_v^{\beta'} B[h_\eps] \cdot
\nabla_v \partial_x^{\alpha''}\partial_v^{\beta''} G_R \partial_x^\alpha \partial_v^\beta G_R \dd x \dd v,
\]
where $C_{\alpha,\beta,\alpha',\beta'}$ is a positive constant depending only on $\alpha$, $\beta$, $\alpha'$, and $\beta'$.

\subsubsection*{$J_2$ terms with $2 \leq |\alpha'|+|\beta'| \leq m$:}
We use the estimates on $B[h_\eps]$ in Lemma \ref{l:abc}, H\"{o}lder's inequality, and the Sobolev embedding to obtain
\begin{equation}\label{e:J-2-bound-most}
\begin{split}
\int_{\Omega_R} & \phi^2 \psi_m^2 \partial_x^{\alpha'}\partial_v^{\beta'} B[ h_\eps] \cdot
\nabla_v \partial_x^{\alpha''}\partial_v^{\beta''} G_R \partial_x^\alpha\partial_v^\beta G_R \dd x \dd v\\
	&\lesssim \int_{\Omega_R} \phi_2 \vv^{2+\gamma}
\| \partial_x^{\alpha'} h_\eps \|_{H^{|\beta'|}_v}
\phi\psi_m |\nabla_v \partial_x^{\alpha''}\partial_v^{\beta''}G_R|
\phi\psi_m |\partial_x^\alpha\partial_v^\beta G_R | \dd x \dd v\\
	&\leq \int_{\Omega_R} \phi_2 \| \partial_x^{\alpha'} h_\eps \|_{H^{|\beta'|}_v}
\phi \| \psi_m \vv^{1+\gamma}\nabla_v \partial_x^{\alpha''}\partial_v^{\beta''} G_R\|_{L^2_v}
\phi \| \psi_m \vv \partial_x^\alpha\partial_v^\beta G_R\|_{L^2_v} \dd x\\
	&\leq \begin{cases}
\| \phi_2 \partial_x^{\alpha'} h_\eps \|_{L^\infty_x H^{|\beta'|}_v}
\| \phi\psi_m \vv^{1+\gamma} \nabla_v \partial_x^{\alpha''}\partial_v^{\beta''} G_R \|_{L^2}
\| \phi\psi_m \vv \partial_x^\alpha\partial_v^\beta G_R \|_{L^2}, &
|\alpha'|+|\beta'| \leq m-2\\
\| \phi_2 \partial_x^{\alpha'} h_\eps \|_{L^4_x H^{|\beta'|}_v}
\| \phi\psi_m \vv^{1+\gamma} \nabla_v \partial_x^{\alpha''}\partial_v^{\beta''} G_R \|_{L^4_x L^2_v}
\| \phi\psi_m \vv \partial_x^\alpha\partial_v^\beta G_R \|_{L^2}, &
|\alpha'|+|\beta'| = m-1\\
\| \phi_2 \partial_x^{\alpha} h_\eps \|_{L^2_x H^{|\beta|}_v}
\| \phi\psi_m \vv^{1+\gamma} \nabla_v G_R \|_{L^\infty_x L^2_v}
\| \phi\psi_m \vv \partial_x^\alpha\partial_v^\beta G_R \|_{L^2}, &
|\alpha'|+|\beta'| = m \geq 3\\
\| \phi_2 \partial_x^{\alpha} h_\eps \|_{L^\infty_x H^{|\beta|}_v}
\| \phi\psi_m \vv^{1+\gamma} \nabla_v G_R \|_{L^2}
\| \phi\psi_m \vv \partial_x^\alpha\partial_v^\beta G_R \|_{L^2}, &
|\alpha'|+|\beta'| = m = 2
\end{cases}\\
&\leq \frac {2 \kappa} N \| G_R \|_{\dot{H}^{m,1}_{\ul,R/2^m}(\Omega_R)}^2 + C_1 (1+X^\Lambda) Z_{m,R}^2,
\end{split}
\end{equation}
through identical analysis as for $J_1$. We emphasize that for each of the cases above,
$|\alpha'|+|\beta'|\geq 2$.

\subsubsection*{$J_2$ terms with $|\alpha'|+|\beta'| = 1$:}
Here there are $m$ derivatives on each of the $G_R$ factors in the integrands.
We again use Lemma \ref{l:abc}, H\"{o}lder's inequality, and the Sobolev embedding to obtain
\begin{equation}\label{e:J-2-bound-(1)}
\begin{split}
\int_{\Omega_R} \phi^2 & \psi_m^2 \partial_x^{\alpha'}\partial_v^{\beta'} B[h_\eps] \cdot
\nabla_v \partial_x^{\alpha''}\partial_v^{\beta''} G_R \partial_x^\alpha\partial_v^\beta G_R \dd x \dd v\\
& \lesssim \int_{\Omega_R} \phi_2 \| \partial_x^{\alpha'} h_\eps \|_{H^{|\beta'|}_v}
\phi \| \psi_m \vv^{1+\gamma}\nabla_v \partial_x^{\alpha''}\partial_v^{\beta''} G_R\|_{L^2_v}
\phi \| \psi_m \vv \partial_x^\alpha\partial_v^\beta G_R\|_{L^2_v} \dd x\\
& \leq \| \phi_2 \partial_x^{\alpha'} h_\eps \|_{L^\infty_x H^{|\beta'|}_v}
\| \phi\psi_m \vv^{1+\gamma} \nabla_v \partial_x^{\alpha''}\partial_v^{\beta''} G_R \|_{L^2}
\| \phi\psi_m \vv \partial_x^\alpha\partial_v^\beta G_R \|_{L^2}\\
&\leq \frac {2 \kappa} N \| G_R \|_{\dot{H}^{m,1}_{\ul,R/2^m}(\Omega_R)}^2 + C_1(1+ X)^{\Lambda} Z_{m,R}^2,
\end{split}
\end{equation}
through identical analysis as above.

\subsubsection*{$J_2$ term with $|\alpha'|+|\beta'|=0$:}
We integrate by parts to write
\begin{equation*}
\begin{split}
\int_{\Omega_R} \phi^2 & \psi_m^2 B[h_\eps] \cdot
\nabla_v \partial_x^{\alpha}\partial_v^{\beta} G_R \partial_x^\alpha\partial_v^\beta G_R \dd x \dd v\\
&= -\frac 1 2 \int_{\Omega_R} \phi^2\psi_m^2 \nabla_v \cdot B[h_\eps]
|\partial_x^\alpha\partial_v^\beta G_R|^2 \dd x \dd v  - \int_{\Omega_R} \phi^2\psi_m \nabla_v \psi_m \cdot B[h_\eps]
|\partial_x^\alpha\partial_v^\beta G_R|^2 \dd x \dd v\\
&=: J_{2,1} + J_{2,2}.
\end{split}
\end{equation*}
By Lemma \ref{l:abc},
\begin{equation}\label{e:J-2-bound-(0)-1}
J_{2,1} \lesssim \int_{\Omega_R} \phi^2\psi_m^2 \| h_\eps \|_{H^1_v} \vv^{2+\gamma}
|\partial_x^\alpha\partial_v^\beta G_R|^2 \dd x \dd v
\leq \frac \kappa N \| G_R \|_{\dot{H}^{m,1}_{\ul,R/2^m}(\Omega_R)}^2 + C_1 X^{\Lambda} Z_{m,R}^2.
\end{equation}
The other term is bounded in the same way as $J_{1,3}$. Namely,
\begin{equation}\label{e:J-2-bound-(0)-2}
\begin{split}
J_{2,2} &\lesssim \int_{\Omega_R} \phi^2 \| h_\eps \|_{L^2_v} \vv^{2+\gamma}
\left( \psi_m^2+|\nabla_v \psi_m|^2 \right)
|\partial_x^\alpha\partial_v^\beta G_R|^2 \dd x \dd v\\
&\leq \frac {2 \kappa} N \| G_R \|_{\dot{H}^{m,1}_{\ul,R/2^m}(\Omega_R)}^2 + C_1 (1+X^\Lambda) Z_{m,R}^2
+ C_1R^\gamma X \| G_R \|_{\dot{H}^{m,0}_{\ul,R/2^{m-1}}(\Omega_R)}^2.
\end{split}
\end{equation}

Combining \eqref{e:J-2-bound-most}, \eqref{e:J-2-bound-(1)}, \eqref{e:J-2-bound-(0)-1},
and \eqref{e:J-2-bound-(0)-2}, we see that
\begin{equation}\label{e:J-2}
J_2 \leq \tilde{C} \frac \kappa N \| G_R \|_{\dot{H}^{m,1}_{\ul,R/2^m}(\Omega_R)}^2
+ C_1 (1+X^\Lambda) Z_{m,R}^2 + C_1 R^\gamma X \| G_R \|_{\dot{H}^{m,0}_{\ul,R/2^{m-1}}(\Omega_R)}^2.
\end{equation}

\subsubsection*{Estimating $J_3$:}
We have
\[
J_3 = \sum_{\substack{\alpha'+\alpha''=\alpha\\ \beta'+\beta'' = \beta}}
\int_{\Omega_R} \phi^2 \psi_m^2  \partial_x^{\alpha'}\partial_v^{\beta'} C[h_\eps] 
\partial_x^{\alpha''}\partial_v^{\beta''} G_R \partial_x^\alpha\partial_v^\beta G_R \dd x \dd v.
\]
Here, the case-by-case analysis is simpler because there is no longer any extra gradient in $v$.
Recall that
$C[h_\eps] = \bar{c}[\mu h_\eps] + \bar{a}_{ij}[\mu h_\eps] \mu^{-1}\partial_{ij} \mu.$
Then, using the expression \eqref{e:mu} for $\partial_{ij} \mu /\mu$, we write
\begin{align*}
\int_{\Omega_R} \phi^2 &\psi_m^2  \partial_x^{\alpha'}\partial_v^{\beta'} C[h_\eps] 
\partial_x^{\alpha''}\partial_v^{\beta''} G_R \partial_x^\alpha\partial_v^\beta G_R \dd x \dd v\\
&\lesssim \int_{\Omega_R} \phi^2\psi_m^2 |\partial_x^{\alpha'}\partial_v^{\beta'} c[\mu h_\eps]|
|\partial_x^{\alpha''}\partial_v^{\beta''} G_R| |\partial_x^\alpha\partial_v^\beta G_R| \dd x \dd v\\
&\quad\quad + \int_{\Omega_R} \phi^2\psi_m^2 |\partial_x^{\alpha'}\partial_v^{\beta'}
\left( v \cdot a[\mu h_\eps] \cdot v \right)|
|\partial_x^{\alpha''}\partial_v^{\beta''} G_R| |\partial_x^\alpha\partial_v^\beta G_R| \dd x \dd v
=: J_{3,1} + J_{3,2}.
\end{align*}

\subsubsection*{$J_{3,1}$ term with $|\alpha'| + |\beta'| \geq 1$:} 
For the first term, in preparation to apply Lemma \ref{l:coeffs}(c) with $\theta = 7/2$, we notice that
\begin{align*}
J_{3,1} &\leq \int_{\Omega_R} \phi_2 \left( \int_{\R^3} \vv^{-\frac 7 2}
| \partial_x^{\alpha'}\partial_v^{\beta'}\bar{c}[\mu h_\eps]|^2 \dd v \right)^{\frac 1 2}
\phi\|\psi_m \vv^{\frac 7 8} G_R \|_{L^\infty_v}
\phi\|\psi_m \vv^{\frac 7 8} \partial_x^\alpha\partial_v^\beta G_R \|_{L^2_v} \dd x\\
  &\qquad \qquad \text{ if } \ |\alpha'|+|\beta'| = m \geq 2,\\
J_{3,1} &\leq \int_{\Omega_R} \phi_2 \left( \int_{\R^3} \vv^{-\frac 7 2}
| \partial_x^{\alpha'}\partial_v^{\beta'}\bar{c}[\mu h_\eps]|^3 \dd v \right)^{\frac 1 3}
\phi\|\psi_m \vv^{\frac 7{12}} G_R\|_{W^{1,6}_v}
\phi\|\psi_m \vv^{\frac 7{12}}\partial_x^\alpha\partial_v^\beta G_R \|_{L^2_v} \dd x\\
 &\qquad \qquad \text{ if } \ |\alpha'|+|\beta'| = m-1 \geq 1,\\
J_{3,1} &\leq \int_{\Omega_R} \phi_2 \left( \int_{\R^3} \vv^{-\frac 7 2}
|\partial_x^{\alpha'}\partial_v^{\beta'} \bar{c}[\mu h_\eps]|^4 \dd v \right)^{\frac 1 4}
\phi\|\psi_m \vv^{\frac 7{16}} G_R\|_{W^{m-2,4}_v}
\phi\|\psi_m \vv^{\frac 7{16}}\partial_x^\alpha\partial_v^\beta G_R \|_{L^2_v} \dd x\\
 &\qquad \qquad \text{ if } \ 2 \leq |\alpha'|+|\beta'| \leq m-2,\\
J_{3,1} &\leq \int_{\Omega_R} \phi_2 \left( \int_{\R^3} \vv^{-\frac 7 2}
| \partial_x^{\alpha'}\partial_v^{\beta'}\bar{c}[\mu h_\eps]|^6 \dd v \right)^{\frac 1 6}
\phi\|\psi_m \vv^{\frac 7{24}} G_R\|_{W^{m-1,3}_v}
\phi\|\psi_m \vv^{\frac 7{24}}\partial_x^\alpha\partial_v^\beta G_R \|_{L^2_v} \dd x\\
 &\qquad \qquad \text{ if } \ |\alpha'|+|\beta'| = 1, \ \text{ any } m.
\end{align*}
Therefore,
\begin{align*}
J_{3,1} &\lesssim \begin{cases}\| \phi_2 h_\eps \|_{H^m} \| \phi\psi_m \vv^{\frac 7 8} G_R \|_{L^\infty}
\| \phi\psi_m \vv^{\frac 7 8} G_R \|_{H^m}  &\text{ if } \ |\alpha'|+|\beta'| = m \geq 2,\\
\|\phi_2 h_\eps \|_{W^{m-1,3}} \| \phi\psi_m \vv^{\frac 7{12}} G_R\|_{W^{1,6}}
\| \phi\psi_m \vv^{\frac 7{12}} G_R \|_{H^m}  &\text{ if } \ |\alpha'|+|\beta'| = m-1 \geq 1,\\
 \| \phi_2 h_\eps \|_{W^{m-2,4}} \| \phi\psi_m \vv^{\frac 7{16}} G_R\|_{W^{m-2,4}}
\| \phi\psi_m \vv^{\frac 7{16}} G_R \|_{H^m}  &\text{ if } \ 2 \leq |\alpha'|+|\beta'| \leq m-2,\\
\| \phi_2 h_\eps \|_{W^{1,6}} \|\phi\psi_m \vv^{\frac 7{24}} G_R\|_{W^{m-1,3}}
\| \phi\psi_m \vv^{\frac 7{24}} G_R \|_{H^m}  &\text{ if } \ |\alpha'|+|\beta'| = 1, \ \text{ any } m.\end{cases}
\end{align*}
By (weighted) Sobolev embedding and Young's inequality with $\eta>0$, we have in all cases
\begin{equation}\label{e:J-3-bound-1}
J_{3,1} \lesssim X \left( \eta \| G_R \|_{\dot{H}^{m,1}_{\ul,R/2^m}(\Omega_R)}
+ C_1 \eta^{-7} Z_{m,R} \right)^2
	\leq \frac {2\kappa}{N} \|  G_R \|_{H^{m,1}_{\ul,R/2^m}}^2 + C_1 (1+X^\Lambda) Z_{m,R}^2,
\end{equation}
where we have chosen $\eta = \kappa^{\frac 1 2} N^{-\frac 1 2} X^{-\frac 1 2}$.

\subsubsection*{$J_{3,1}$ term with $|\alpha'| + |\beta'| = 0$:}
From Lemma \ref{l:coeffs}(b) with $p=\infty$ and the Sobolev embedding ($H^4(\R^6) \subset L^\infty(\R^6)$), we see that
\begin{equation}\label{e:J-3-bound-2}
J_{3,1}
	\lesssim \|\phi_2 h_\eps \|_{L^\infty} \int_{\Omega_R} \phi^2\psi_m^2 \vv^\gamma 
|\partial_x^\alpha \partial_v^\beta G_R|^2 \dd x \dd v
	\leq (1+X^\Lambda) Z_{m,r}^2
\end{equation}

\subsubsection*{$J_{3,2}$ term:}
For the last term, Lemma \ref{l:coeffs}(a) with $p=2$ implies
\begin{equation}\label{e:J-3-bound-3}
\begin{split}
J_{3,2} &\lesssim \| \phi_2 h_\eps \|_{H^m} \int_{\Omega_R} \phi^2\psi_m^2
| \vv \partial_x^\alpha \partial_v^\beta G_R |
| \vv^{1+ \gamma} \partial_x^\alpha\partial_v^\beta G_R | \dd x \dd v\\
&\leq \frac{2 \kappa} N \| G_R \|_{\dot{H}^{m,1}_{\ul,R/2^m}(\Omega_R)}^2 + C_1 (1+X^\Lambda) Z_{m,R}^2.
\end{split}
\end{equation}

Combining \eqref{e:J-3-bound-1}-\eqref{e:J-3-bound-3}, we see that
\begin{equation}\label{e:J-3}
J_3 \leq \tilde{C} \frac \kappa N \| G_R \|_{\dot{H}^{m,1}_{\ul,R/2^m}(\Omega_R)}^2
+ C_1 (1+X^\Lambda) Z_{m,R}^2 + C_1 R^\gamma X \| G_R \|_{\dot{H}^{m,0}_{\ul,R/2^{m-1}}(\Omega_R)}^2.
\end{equation}

\subsubsection*{Estimating $J_4$ and $J_5$:} All derivatives appearing in $J_4$ are of order at most $m$.
By Cauchy-Schwartz,
\begin{equation}\label{e:J-4}
 J_4 \lesssim \int_{\Omega_R} \phi^2\psi_m^2 \vv |D_{x,v}^m G_R | |\partial_x^\alpha\partial_v^\beta G_R|\dd x \dd v
\leq \frac \kappa N \| G_R \|_{\dot{H}^{m,1}_{\ul,R/2^m}(\Omega_R)}^2 + C_1 Z_{m,R}^2.
\end{equation}
Likewise, since $\nabla_x \phi \lesssim \phi_2$, we integrate by parts in $J_5$ to obtain
\begin{equation}\label{e:J-5}
\begin{split}
 J_5 &= \int_{\Omega_R} \psi_m^2\phi \nabla_x \phi\cdot v |\partial_x^\alpha\partial_v^\beta G_R|^2 \dd x \dd v
\lesssim \int_{\Omega_R} \phi\phi_2\psi_m^2 \vv |\partial_x^\alpha\partial_v^\beta G_R|^2 \dd x \dd v\\
&\leq \frac \kappa N \| G_R \|_{\dot{H}^{m,1}_{\ul,R/2^m}(\Omega_R)}^2 + C_1 Z_{m,R}^2.
\end{split}
\end{equation}

\subsubsection*{Estimating $J_6$:}
For the final quantity on the right-hand side of \eqref{e:main-est}, we integrate by parts to write
\[
 J_6 = -\eps \int_{\Omega_R} \phi^2\psi_m^2 |\nabla \partial_x^\alpha\partial_v^\beta G_R|^2 \dd x \dd v
-\eps \int_{\Omega_R} \nabla \left( \phi^2\psi_m^2 \right) \cdot \nabla \partial_x^\alpha\partial_v^\beta G_R
\partial_x^\alpha\partial_v^\beta G_R \dd x \dd v.
\]
The first term is negative. We will need it to close the estimates. We integrate the second term by
parts once more (note that $\supp(\phi \psi_m) \subset \Omega_R$ for $m \geq 1$),
yielding
\begin{equation*}
\begin{split}
-\eps \int_{\Omega_R} &\nabla \left( \phi^2\psi_m^2 \right) \cdot \nabla \partial_x^\alpha\partial_v^\beta G_R
\partial_x^\alpha\partial_v^\beta G_R \dd x \dd v
= \frac \eps 2 \int_{\Omega_R} \Delta \left( \phi^2\psi_m^2 \right)
|\partial_x^\alpha\partial_v^\beta G_R|^2 \dd x \dd v\\
&\lesssim \eps \int_{\Omega_R} \left( \phi^2\psi_m^2 + \psi_m^2(|\nabla_x \phi|^2 + \phi |D^2_x \phi|)
+ \phi^2(|\nabla_v \psi_m|^2 + \psi_m |D^2_v \psi_m|) \right)
|\partial_x^\alpha\partial_v^\beta G_R|^2 \dd x \dd v\\
&\leq C_1 \eps Z_{m,R}^2 + C_1 \frac \eps {R^2} \| G_R \|_{\dot{H}^{m,0}_{\ul,R/2^{m-1}}(\Omega_R)}^2.
\end{split}
\end{equation*}
For the last inequality, we used analysis similar to $J_{1,3}$ above.
We also used that $|\nabla_v \psi_m|$, $|D^2_v \psi_m| \lesssim R^{-2} \psi_{m-1}$.
Therefore,
\begin{equation}\label{e:J-6}
J_6 + \eps \int_{\Omega_R} \phi^2\psi_m^2 |\nabla \partial_x^\alpha\partial_v^\beta G_R|^2 \dd x \dd v
\leq C_1 \eps Z_{m,R}^2 + C_1 \frac \eps {R^2} \| G_R \|_{\dot{H}^{m,0}_{\ul,R/2^{m-1}}(\Omega_R)}^2.
\end{equation}

\subsubsection*{Proof of \eqref{e:induction-bound}:}
Combining \eqref{e:J-1}, \eqref{e:J-2}, \eqref{e:J-3}, \eqref{e:J-4}, \eqref{e:J-5}, and \eqref{e:J-6},
then summing over all $\alpha$ and $\beta$ with $|\alpha|+|\beta| = m$ yields
\begin{equation}\label{e:almost-done}
\begin{split}
\frac 1 2 \frac d{dt} &\| G_R \|_{\dot{H}^{m,0}_{\ul,R/2^m}(\Omega_R)}^2
+\kappa \| G_R \|_{\dot{H}^{m,1}_{\ul,R/2^m}(\Omega_R)}^2
+\eps \| G_R \|_{\dot{H}^{m+1,0}_{\ul,R/2^m}(\Omega_R)}^2\\
&\leq \tilde{C} \frac \kappa N \| G_R \|_{\dot{H}^{m,1}_{\ul,R/2^m}(\Omega_R)}^2
+ C_1 (1+X^\Lambda) Z_{m,R}^2
+ C_1 R^\gamma X \| G_R \|_{\dot{H}^{m,0}_{\ul,R/2^{m-1}}(\Omega_R)}^2.
\end{split}
\end{equation}
If we now choose $N$ larger than $2\tilde{C}/\kappa$, the first term on the right-hand side is absorbed
on the left. We then define
\[ \tilde{G}_m = C_1 R^\gamma X \| G_R \|_{\dot{H}^{m,0}_{\ul,R/2^{m-1}}(\Omega_R)}^2. \]
By induction, \eqref{e:induction-bound} holds for $m-1$. Integrating from $0$ to $t$, we have
\begin{align*}
 \| G_R(t) \|_{\dot{H}^{m-1,0}_{\ul,R/2^{m-1}}(\Omega_R)}^2
&+ \kappa \int_0^t \| G_R(s) \|_{\dot{H}^{m-1,1}_{\ul,R/2^{m-1}}(\Omega_R)}^2 \dd s
+ \eps \int_0^t \| G_R(s) \|_{\dot{H}^{m,0}_{\ul,R/2^{m-1}}(\Omega_R)}^2 \dd s\\
&\leq \| g_{in} \|_{H^k_{\ul}}^2 + C_1 \int_0^t \left( 1 + X(s)^\Lambda \right)
Z_{m-1,R}^2(s) \dd s + \eps.
\end{align*}
Applying Gr\"{o}nwall's inequality, we see that for all $t \in [0,T]$
\begin{align*}
\| G_R \|_{\dot{Y}^{m-1}_{t,R/2^{m-1},R}}^2
+ \eps \int_0^t \| G_R(s) \|_{\dot{H}^{m,0}_{\ul,R/2^{m-1}}(\Omega_R)}^2 \dd s
\leq \left( \| g_{in} \|_{H^k_{\ul}}^2 + \eps \right)
\exp \left( C_1 \int_0^t \left( 1+X(s)^\Lambda \right) \dd s \right).
\end{align*}
See \eqref{e:locnorm}. In particular, the bound on the second term implies that
\begin{equation}\label{e:R-eps-condition}
 \int_0^{T} \tilde{G}_m(s) \dd s \leq \eps \ \text{ if } \
R \geq 2 \left( \frac {C_1 X}{\eps^2} \left( \| g_{in} \|_{H^k_{\ul}}^2 + \eps \right)
\exp \left( C_1 T (1+X^\Lambda) \right) \right)^{- \frac 1 \gamma}.
\end{equation}
Fixing $R$ as above, and using the bound from \eqref{e:R-eps-condition} in \eqref{e:almost-done} yields \eqref{e:induction-bound}.

\subsubsection*{Conclusion of Proof:}
Summing \eqref{e:base-bound} and \eqref{e:induction-bound} for all $m$ up to $k$ and
using Gr\"{o}nwall's inequality as above gives, for all $t \in [0,T_{\rho_0,\kappa}]$,
\begin{equation}\label{e:gronwall}
\| G_R \|_{Y^k_{t,R}}^2 \leq \left(\| g_{in} \|_{H^k_{\ul}}^2 + \eps \right)
\exp \left( C_1 T (1+X(t)^\Lambda) \right).
\end{equation}

Now consider the sequence of functions $\lbrace{ G_K \rbrace}$ (for $K \in \N$)
of solutions to the problem \eqref{e:glin} on $\Omega_K$ with boundary conditions \eqref{e:ginit}
and with $\eps = (\ln K)^{-1}$. Note that this choice of $\eps$ still allows condition
\eqref{e:R-eps-condition} to hold true, for sufficiently large $K$ depending on the parameters of the
problem.

The bound in \eqref{e:gronwall} holds for each such $G_K$. If $L>0$ is any large
number, we conclude that $\| \chi_L G_K \|_{Y^k_T}$ is bounded uniformly in $K$.
Recall that $\chi_L$ is a smooth cutoff function in $x$ and $v$, supported in the ball of radius $L-1$,
and equal to $1$ in the ball of radius $L-2$. Therefore, a subsequence converges weakly to some
limit in $\bar{G}_L \in Y^k_T$ supported on the ball of radius $L-1$. Note that
$\bar{G}_L$ and $\bar{G}_{L'}$ are identical on the ball of radius $\min({L,L'})$ and
for all $t \in [0,T]$.

A diagonalization argument allows us to take $L$ to $\infty$ and extract a subsequence (which we also denote
as $\lbrace{ G_K \rbrace}$) and a limit $\bar{G} \in Y^k_T$ such that
\[
 G_K \rightarrow \bar{G} \ \text{ in } Y^k_T \ \text{ on compact sets.}
\]
Lemma \ref{l:abc} implies that $A[h_\eps] \to A[h]$ in the space $L^\infty([0,T],H^k_{\ul})$ as $\eps = (\ln K)^{-1}\to 0$, and similarly for $B[h_\eps]$ and $C[h_\eps]$. Since $k \geq 4$, the function $\bar{G}$ has sufficient regularity that it is a solution
to the linearized problem \eqref{e:true-glin} in $W^{1,\infty}_t H^2_{x,v}$ on all of $\R^6$.
By the maximum principle for \eqref{e:glin}, each $G_K$ is nonnegative,
and therefore so is $\bar{G}$. Furthermore, $\bar{G}$ inherits the
bound \eqref{e:main-bound} from \eqref{e:gronwall}.
\end{proof}

We are now ready to solve \eqref{e:g} by constructing a sequence of approximate solutions in the space
$Y_T^k$ given by \eqref{eq:Y_definition}.

\begin{theorem}\label{t:gh}
Assume that $g_{in} \in H^k_{\ul}$, and that
\[\|g_{in}\|_{H^k_{\ul}} \leq M_0.\]
Then, for some $T\in (0,T_{\rho_0,\kappa}]$ depending on $M_0$, there exists a unique nonnegative
$g \in Y_T^k$ solving \eqref{e:g} with $g(0,x,v) = g_{in}(x,v)$.
\end{theorem}

We emphasize that, although $T$ depends on $M_0$, $M_0$ can be arbitrarily large.

\begin{proof}
Define $g^0(t,x,v) = g_{in}(x,v)$ and, for $n \geq 1$, define the sequence $\lbrace{ g^n \rbrace}$
recursively as the solution of
\begin{equation}\label{e:iterate}
\partial_t g^n + v \cdot \nabla_x g^n + \kappa \vv g^n =
\text{tr} \left( A[g^{n-1}] D_v^2 g^n \right) + B[g^{n-1}] \cdot \nabla_v g^n + C[g^{n-1}] g^n,
\end{equation}
with $g^n(0,x,v) = g_{in}(x,v)$. This is precisely the linearized problem \eqref{e:true-glin}. Then,
by Lemma \ref{l:gre}, for any $T \in (0,T_{\rho_0,\kappa}]$, each $g^n$ exists, is nonnegative, belongs to $Y^k_T$,
and satisfies
\begin{equation}\label{e:gn-bound}
\| g^n \|^2_{Y^k_T} \leq \| g_{in} \|_{H^k_{\ul}}^2
\exp \left( C_1 T \left(1 + \| g^{n-1} \|_{L^\infty([0,T],H^k_{\ul})}^\Lambda \right) \right),
\end{equation}
for some $\Lambda > 1$ and $C_1>0$ that are independent of $n$.

Assume by induction that, for $n \geq 1$,
\begin{equation}\label{e:induct}
\|g^{n-1}\|_{L^\infty([0,T],H^k_{\ul})} \leq 2M_0,
\end{equation}
for some $T\in (0,T_{\rho_0,\kappa}]$. This hypothesis holds for $n=1$ by our assumption on $g_{in}$.
Then \eqref{e:gn-bound} becomes
\[\|g^n\|_{Y_T^k}^2 \leq M_0^2 \exp(C_1 T(1+(2M_0)^\Lambda).\]
If we take
\[ T \leq \min\left\{\frac {2\ln 2}{C_1(1+(2M_0)^\Lambda)}, T_{\rho_0,\kappa}\right\},  \]
then $\|g^n\|_{Y_T^k} \leq 2M_0$, and in particular $\|g^n\|_{L^\infty([0,T],H^k_{\ul})} \leq 2M_0$.
Note that $T$ is independent of $n$.
We conclude \eqref{e:induct} holds for all $n\geq 1$.

Next, define $w^n = g^n - g^{n-1}$. Equation \eqref{e:iterate} implies, for $n \geq 2$,
\begin{align*}
\partial_t w^n + v\cdot \nabla_x w^n + \kappa \vv^2 w^n &=
\tr\left( A[g^{n-1}]D_v^2 w^n \right) + B[g^{n-1}] \cdot \nabla_v w^n + C[g^{n-1}] w^n\\
&\quad \quad + \tr\left( A[w^{n-1}]D_v^2 g^{n-1} \right) + B[w^{n-1}] \cdot \nabla_v g^{n-1}
+ C[w^{n-1}] g^{n-1},
\end{align*}
and $w^n(0,x,v) = 0$. For all multi-indices with $|\alpha|+|\beta| \leq k$, we differentiate the equation
for $w^n$ by $\partial_x^\alpha \partial_v^\beta$, multiply by $\phi^2 \partial_x^\alpha \partial_v^\beta w^n$,
and integrate over $\R^6$. Note that the estimates developed in the proof of Lemma \ref{l:gre} were
independent of $R$ and $\eps$. Repeating the calculations (now without any cutoff in $v$ or mollification of $g^{n-1}$ or $w^{n-1}$) yields,
by Lemmas \ref{l:coeffs} and \ref{l:abc} and the fact that
$\|g^n\|_{Y_T^k}\leq 2M_0$ for all $n\geq 0$,
\begin{align*}
\|w^n\|_{Y_T^k}^2 &\leq C_1\int_0^T\left(1+\|g^{n-1}(s)\|_{H^k_{\ul}}^\Lambda \right)
\|w^n(s)\|_{H^k_{\ul}}^2\dd s\\
&\quad \quad + C_1 \int_0^T\|w^{n-1}(s)\|_{H^k_{\ul}}^\Lambda
\left(\|g^{n-1}(s)\|_{H^k_{\ul}}^2 + \|w^n(s)\|_{H^k_{\ul}}^2\right)\dd s\\
&\leq C_1 T (1+(2M_0)^\Lambda)\|w^n\|_{Y_T^k}^2 + C_1 T \left((2M_0)^\Lambda
+ (4M_0)^\Lambda \right)(4M_0)^{(\Lambda-2)_+}\|w^{n-1}\|_{Y_T^k}^{2},
\end{align*}
since $\|w^{n}\|_{H^k_{\ul}}\leq 4M_0$ and similarly for $w^{n-1}$.  
If necessary, we choose $T$ smaller, so that 
\[
C_1 T \left( 1+(2M_0)^\Lambda \right) \leq \frac 1 2 \quad \mbox{ and } \quad C_1 T
\left( (2M_0)^\Lambda + (4M_0)^\Lambda \right)(4M_0)^{(\Lambda-2)_+} \leq \frac 1 4.\] 
Now we have
\begin{equation}\label{e:12}
\|g^n-g^{n-1}\|_{Y_T^k} \leq \frac 1 2 \|g^{n-1}-g^{n-2}\|_{Y_T^k}.
\end{equation}
We conclude $g^n$ is a convergent sequence and the limit $g \in Y_T^k$ is a classical solution of \eqref{e:g}.

The uniqueness of $g$ follows along the same lines. If $g_1$ and $g_2$ are two solutions of \eqref{e:g}
in $Y_T^k$ with the same initial data, then $w := g_1 - g_2$ satisfies
\begin{equation*}
\begin{split}
\partial_t w + v \cdot \nabla_x w + \kappa \vv^2 w &=
\text{tr} \left( A[g_2] D_v^2 w \right) + B[g_2] \cdot \nabla_v w + C[g_2] w\\
&\quad \quad +\text{tr} \left( A[w] D_v^2 g_2 \right) + B[w] \cdot \nabla_v g_2 + C[w] g_2,
\end{split}
\end{equation*}
and $w(0,x,v) = 0$. By the same estimate as above, and Gr\"{o}nwall's inequality, we conclude that
$\| w \|_{Y^k_T} = 0$.
\end{proof} 

Theorem \ref{t:gh} implies our first main result, Theorem \ref{t:LWP}, with $f = e^{-(\rho_0-\kappa t)\vv^2} g$.

\section{Mass-spreading}\label{s:mass}

We first state a slightly weakened form of \Cref{thm:mass}:
\begin{proposition}\label{prop:mass_weak}
	Suppose that the assumptions of \Cref{thm:mass} hold.  Suppose further that there exist $\delta_0,r_0>0$ and $x_0, v_0 \in \R^3$ such that
	\begin{equation}\label{eq:mass_core}
		\delta_0 \1_{B_{r_0}(x_0) \times B_{r_0}(v_0)}
			\leq f_{in}.
	\end{equation}
	Then there exists $T_* > 0$, depending only on $r_0$ and the upper bound of the physical quantities in~\eqref{eq:physical} such that 
 for every $0 < \underline T \leq t \leq T_*$, there exists $\nu>0$, which depends only on $\delta_0$, $r_0$, $\underline T$, $\overline T$, $|v_0|$, the physical quantities in~\eqref{eq:physical}, and $|x-x_0|$, and $\rho>0$, which depends on the same quantities, such that
\begin{equation}\label{eq:weak_lower_bound}
	f(t,x,v) \geq \nu \exp\left\{ - \rho |v|^{\max\{4,3-\gamma\}}\right\}.
\end{equation}

\end{proposition}

We first prove \Cref{prop:mass_weak}, and then show how to obtain \Cref{thm:mass} from it.


\subsection{Proof of \Cref{prop:mass_weak}}


We prove the lower bound in \Cref{prop:mass_weak} using a probabilistic representation.
In order to do this, we require an approximation process; namely, we need to cut-off $\bar a_{ij}$ at
infinity and regularize $f$. The former is to construct a unique stochastic process associated to the
equation and the latter is to apply Ito's lemma and relate the stochastic process to $f$.

We define this approximation process now. Let $\chi$ be any smooth cut-off function such that $0 \leq \chi \leq 1$, $\chi(v) = 1$ if $|v| \leq 1$, and $\chi(v) = 0$ if $|v| \geq 2$.  For any $R$ sufficiently large, define
\begin{equation}\label{eq:aR}
	\bar a_R[f](t,x,v) = \chi(v/R)\bar a[f](t,x,v)  + (1-\chi(v/R)) I,
\end{equation}
where $I$ is the identity matrix on $\R^3$. 
Then, for any $\eps \in(0,1)$, define $f_{R,\eps}$ as the solution to 
\begin{equation}\label{e:nondivergence_regularized}
\begin{cases}
\partial_t f_{R,\eps} + v\cdot \nabla_x f_{R,\eps} =\tr\left[(\bar a_R[f] +
\eps I)D_v^2 f_{R,\eps}\right] + \bar c[f] f_{R,\eps}, \quad &\text{ in } (0,T]\times
\R^3\times\R^3,\\
f_{R,\eps} = f_{in}, &\text{ on } \{t=0\} \times\R^3\times\R^3,
\end{cases}
\end{equation}
With $f$ fixed in the coefficients, the existence and uniqueness of $f_{R,\eps}$ follows from the work in Section \ref{s:LWP}.  Indeed,~\eqref{eq:aR} is the linear Landau equation, so the bounds on $f_{R,\eps}$ are, in fact, easier to obtain.  We get, immediately, that $e^{\rho \vv^2} f_{R,\eps}$ is bounded in $Y^4_T$ independently of $R$ and $\eps$.

Next, we claim that $f \in C^\alpha ([0,T]\times \R^3 \times \R^3)$ for some $\alpha \in (0,1)$.  Indeed, since $f \in L^\infty_t H^4_{x,v} \cap W^{1,\infty}_t H^2_{x,v}$, then, by the Sobolev inequality, $f\in L^p_t W^{1,p}_{x,v} \cap W^{1,6}_t L^6_{x,v}$ for any $p \in [1,\infty)$.  We may then apply the anisotropic Sobolev embedding~\cite[Theorem~2]{haskovec2009} with $p > 36/5$ to obtain the H\"older continuity of $f$.  We note that $f_{R,\e}$ inherits the same bound.

Due to the above discussion, along with the uniqueness of solutions to the linearized Landau equation in the class of functions with $e^{\rho \vv^2} f \in Y^4_T$, which follows from Theorem \ref{t:LWP}, we obtain in particular that
\begin{equation}\label{eq:f_R_to_f}
	\lim_{\eps\to0}\lim_{R\to\infty} f_{R,\eps} = f,
\end{equation}
where the above limit holds locally uniformly in $C^\alpha$, for some $\alpha \in (0,1)$.  This convergence is a key point in our argument since, in general, we obtain pointwise lower bounds for $f_{R,\eps}$ when $R$ is sufficiently large and then we take the limits $R\to \infty$ and $\eps \to 0$ to obtain lower bounds on $f$.

The main tool in the proof of the mass-pushing theorem is a probabilistic interpretation
of~\eqref{e:divergence}. In preparation for this, we set some notation and collect a few important facts.

Since $\bar a_R$ is symmetric and non-negative definite, we may find a symmetric, positive definite matrix $\bar \sigma_{R,\eps}$ such that
\[
	\bar a_R + \eps I = \bar\sigma_{R,\eps}\bar\sigma_{R,\eps}.
\]
%
We note that the upper bound on $\bar a$, i.e. that $\bar a \lesssim \vv^{2+\gamma}$ (cf.~\cite[Appendix A]{henderson2017smoothing}), which depends only on the physical quantities~\eqref{eq:physical}, yield the upper bound $\sigma_{R,\eps}\lesssim \vv^{\max\{0,1+\gamma/2\}}$, which is independent of $R$ and $\eps$.  This is important in the proof of~\Cref{prop:mass_weak}.

Further, since $f\in Y^4_T$, the matrix $\bar a_R$ is uniformly Lipschitz in all variables due to the cut-off in $v$ (see Lemma~\ref{l:coeffs}).  Since $\eps > 0$, we notice that $\bar\sigma_{R,\eps}$ is uniformly Lipschitz as well.  We note that the $\eps$ is not crucial here since the square root of a non-negative $C^2$ function is $C^{0,1}$.  The bound on the Lipschitz constant of $\bar \sigma_{R,\eps}$ depends on $\eps$ and $R$, but we use it only to guarantee the existence of a solution to our stochastic differential equation below.  Importantly, we do not use this Lipschitz bound anywhere in our estimate of $m$.


\begin{lemma}\label{lem:probabilistic}
	Suppose that the conditions of \Cref{thm:mass} hold.  For any $(t,x,v) \in [0,T]\times\R^3\times\R^3$, there is a unique solution to the stochastic differential equation
	\begin{equation}\label{eq:sde}
		\begin{cases}
			\dd V_s^{t,x,v} = \bar\sigma_{R,\eps}\left(t-s, X_s^{t,x,v}, V_s^{t,x,v}\right) \dd W_s,\\
			\dd X_s^{t,x,v} = -V_s^{t,x,v} \dd s,\\
			V_0^{t,x,v} = v, \quad X_0^{t,x,v} = x,
		\end{cases}
	\end{equation}
	for all $s \in (0,t)$, where $W_s$ is a Brownian motion in $\R^3$.  Further, we have 
	\begin{equation}\label{eq:prob_lower_bound}
		f_{R,\eps}(t,x,v)
			=\E\left[ e^{\int_0^t \bar c(t-s,X_s^{t,x,v},V_s^{t,x,v})ds} f_{in}\left(X_t^{t,x,v}, V_t^{t,x,v}\right) \right]
	\end{equation}
\end{lemma}
Though the kinetic setting is non-standard, the proof of \Cref{lem:probabilistic} follows from the usual
arguments. The existence is due to a Picard iteration, the uniqueness is due to Gr\"onwall's inequality,
and the formula for $f_{R,\eps}$ requires only an application of Ito's Lemma.  As such, we omit the proof.  We refer the interested reader to~\cite{KaratzasShreve}.  It is important for the application of Ito's Lemma that $f_{R,\eps}$ be $C^1$ in $t$ and $x$ and $C^2$ in $v$.  Using the Schauder estimates of~\cite{henderson2017smoothing} along with the positive definiteness of $\bar a_R + \eps I$, we see that this is the case.  This is where the $\eps I$ term is crucial; the regularity of $f_{R,\eps}$ is required below.

We also need the following lemma, which shows that pointwise lower bounds for $f$ in a small ball give a lower ellipticity constant for the matrix $\bar a[f]$.  This implies a lower bound for $\bar \sigma_{R,\eps}$. The proof of this lemma is similar to calculations that appeared in \cite{desvillettes2000landau} and \cite{silvestre2015landau}, but there is a key difference: with pointwise lower bounds available, there is no need to use the upper bound on the entropy density $H(t,x)$, so this lemma allows us to remove the entropy assumption from our criteria for smoothness and continuation.

\begin{lemma}\label{l:lower}
	Let $g:\R^3\to\R_+$ be an integrable function such that $g \geq \delta \1_{B_r(v_0)},
	$
	for some $\delta, r>0$ and $v_0 \in \R^3$. Then $\bar a[g]$ defined by \eqref{e:a} satisfies 
	\begin{equation}\label{e:a_lower}
	 \bar a_{ij}[g](t,x,v) e_ie_j \geq c \delta \begin{cases}  (1+ |v|)^{\gamma}, & e\in \mathbb S^2,\\
	 (1+ |v|)^{\gamma+2}, & e\cdot v = 0,\end{cases}
	\end{equation}
	for unit vectors $e$, where $c>0$ is a constant depending only on $\gamma$, $v_0$, and $r$. 
\end{lemma} 
\begin{proof}
We consider the case $v_0 = 0$, but the general case follows similarly.   For any $e\in \mathbb S^2$, \eqref{e:a} implies 
	\begin{equation}\label{eq:c1001}
	\bar a_{ij}[g] e_i e_j = a_{\gamma} \int_{\R^3} \left(1 - \frac{(w\cdot e)^2}{|w|^2}\right) |w|^{\gamma+2} g(v-w) \dd w
	\gtrsim \delta \int_{B_r(v)} \sin^2 (\theta_{e,w})|w|^{\gamma+2} \dd w,
	\end{equation}
	where $\theta_{e,w}$ is the angle between $e$ and $w$. Let $A_\e = \{w \in B_r(v) : |w \cdot e|^2 \geq |w|^2(1 - \e)\}$. Since $\sin \theta_{e,w}$ is close to zero in $A_{\eps}$, we want to avoid that set to derive a lower bound.  We can assume $e = v/|v|$, since that is the worst case, i.e. the case where $A_\eps$ is largest for a given $\eps$. Clearly, there exists $\eps_0 \in(0,1)$ such that $|A_{\e_0}| = |B_r|/2$. With \eqref{eq:c1001}, this already implies $\bar a_{ij}[g]e_i e_j \gtrsim\delta(1+|v|)^\gamma$ for small $|v|$. For  $|v|$ large (compared to $r$), since $A_{\eps_0}$ contains a cylinder of diameter $\approx \sqrt{\eps_0}|v|$ and height $2r$, and this cylinder must have volume bounded independently of $|v|$, we conclude $\eps_0 \approx  |v|^{-2}$.
	Hence, from~\eqref{eq:c1001},
	\[
	\bar a_{ij}[g]e_ie_j
	\gtrsim  \delta \int_{B_r(v) \setminus A_{\e_0}} \e_0 |w|^{\gamma + 2} dw
	\gtrsim \delta \e_0 |v|^{\gamma+2} \frac{|B_r(v)|} 2  
	\gtrsim \delta  |v|^{-2}|v|^{\gamma+2}
	\gtrsim \delta  |v|^\gamma.
\]

	
	To conclude the proof, we need only consider the case when $e \cdot v = 0$ and $|v| \geq 3r$.  In this case, for any $w\in B_r(v)$, we have $w\cdot e \leq r$ and $|w| \geq |v|-r \geq 2|v|/3$.  Thus, $(w\cdot e)^2/|w|^2 \leq \frac 1 3$, which, in turn, implies $\sin^2\theta_{e,w} \geq \frac 4 5$.  From~\eqref{eq:c1001}, we have, as desired,
	\[
		\bar a_{ij}[g] e_i e_j \gtrsim \delta \frac 13 (|v|-r)^{\gamma+2}|B_r(v)| \gtrsim \delta r^3 |v|^{\gamma+2}.
	\]
%
%
\end{proof}


We are now in a position to prove \Cref{prop:mass_weak}.
\begin{proof}[Proof of \Cref{prop:mass_weak}]
We prove this in four steps. 
We first show that the initial bound is preserved on a slightly smaller set for all $t\in[0, T_*]$.
Then we crucially use Lemma \ref{l:lower} to conclude 
the diffusivity matrix $\bar \sigma_{R,\eps}$ is positive definite for $t\in[0, T_*]$ and $x$ near $x_0$. 
This allows us to spread mass to any $v$.  In the third step, we use the fact that we
have mass at any velocity to use the pure transport term to spread mass to any $x$, though only for a small
range in $v$.  Finally, for a fixed $x$, we use the small mass clustered at some of the velocities to
repeat our earlier step: we obtain a lower bound on the viscosity and spread mass to all $v$.  These four steps give us a lower bound on $f$ for all $t\in[0,T_*]$.  

Before beginning, we asssume, without loss of generality, that $\delta_0 \leq 1$.  We also note that all estimates depend on $\gamma$ and the physical quantities~\eqref{eq:physical}, though we often do not mention this dependence explicitly in the sequel.  Finally, we denote $m(t,x,v)$ to be a positive function satisfying the properties as in the statement of Proposition~\ref{prop:mass_weak} that changes line-by-line.

\subsubsection*{Step 1: Preserving a mass core for short times:} The first step is showing that the lower bound on the mass at $(x_0,v_0)$ given by~\eqref{eq:mass_core} remains for a short time.  

We make this explicit.  Fix any $R \geq 2(|v_0| + r_0)$ and $\eps \in(0,1)$.  Let $\underline r_0 = \min\left\{r_0/2, \sqrt r_0\right\}$.  We claim that there exists $T_*>0$ depending only on $v_0$ and $r_0$ such that, for all $t \in [0,T_*]$,
\begin{equation}\label{eq:mass_core_time1}
		f_{R,\eps}(t,x,v) \geq \frac{\delta_0}{2} \1_{B_{\underline r_0}(x_0)\times B_{\underline r_0}(v_0)}(x,v).
\end{equation}
To see this, define $\tau_{\underline r_0} = \inf\{t \in [0,T] : |V_t^{t,x,v} - v_0| > \underline r_0\}$ and use \Cref{lem:probabilistic} to obtain
\begin{equation}\label{eq:chris1}
\begin{split}
	f_{R,\eps}(t,x,v)
		&\geq \E\left[ \1_{\{\tau_{\underline r_0} > t\}} f_{in}(X_t^{t,x,v}, V_t^{t,x,v})\right]
		\geq \delta_0\E\left[ \1_{\{\tau_{\underline r_0} > t\}} \1_{\{X_t^{t,x,v} \in B_{r_0}(x_0)\}} \1_{\{V_t^{t,x,v} \in B_{r_0}(v_0)\}}\right]\\
		&= \delta_0\PP\left\{\tau_{\underline r_0} > t, X_t^{t,x,v} \in B_{r_0}(x_0), V_t^{t,x,v} \in B_{r_0}(v_0) \right\}
\end{split}
\end{equation}
We make the following crucial observation.  By our choice of $\tau_{\underline r_0}$, we consider only trajectories $V_s^{t,x,v}$ that never leave $B_{\underline r_0}(v_0) \subset B_R(0)$.  As a result, $(X_s^{t,x,v}, V_s^{t,x,v})$ and thus our estimates in this step are independent of $R$ (cf.~\eqref{eq:sde} and the definition of $\bar \sigma_{R,\eps}$).

Let $T_2 > 0$ be a constant to be determined and define
\begin{equation}\label{eq:T_0}
	T_* := \min\left\{T,\underline r_0, \frac{1}{2}\frac{r_0 - \underline r_0}{|v_0| + \underline r_0}, T_2\right\}.
\end{equation}
If $|v-v_0|, |x-x_0|\leq \underline r_0$ and $t \in[0,T_*]$ then we claim that
\begin{equation}\label{eq:chris2a}
\begin{split}
	\PP\left\{\tau_{\underline r_0} > t, X_t^{t,x,v} \in B_{r_0}(x_0), V_t^{t,x,v} \in B_{r_0}(v_0) \right\}
		&\geq \PP(\tau_{\underline r_0} > t\}
		= \PP\left\{\max_{0 \leq s \leq t} |V_s^{t,x,v} - v| < \underline r_0\right\}.
\end{split}
\end{equation}
Indeed, suppose that $|V_s^{t,x,v} - v| < \underline r_0$ for all $s\in[0,T_*]$.
First, we observe that $|V_s^{t,x,v}| \leq |v| + \underline r_0 \leq |v_0| + r_0 < R$.  Hence $\tau_{\underline r_0} > t$.  Second,
\[
	|X_t^{t,x,v} - x_0|
		\leq \left|\int_0^t V_s^{t,x,v} \dd s\right| + |x-x_0|
		< \int_0^t \left(|v_0| + \underline r_0\right)\dd s + \underline r_0
		\leq T_* |v_0| + (T_* +1) \underline r_0
		\leq r_0.
\]
Here we used the third term in the definition ~\eqref{eq:T_0} of $T_*$ along with the fact that $\underline r_0 \leq r_0/2$. 
Third, it is clear that $V_t^{t,x,v} \in B_{r_0}(v_0)$.  Hence,~\eqref{eq:chris2a} follows.

From~\eqref{eq:chris2a}, we conclude that
\begin{equation}\label{eq:chris2}
	\PP\left\{\tau_{\underline r_0} > t, X_t^{t,x,v} \in B_{r_0}(x_0), V_t^{t,x,v} \in B_{r_0}(v_0) \right\}
		\geq 1 - \PP\left\{\max_{0 \leq s \leq t} |V_s^{t,x,v} - v| \geq \underline r_0\right\}.
\end{equation}
Hence,~\eqref{eq:mass_core_time1} follows from an upper bound of the last term on the right-hand side of~\eqref{eq:chris2}.  To obtain such an upper bound we first apply Markov's inequality:
\[
	\PP\left\{\max_{0 \leq s \leq t} |V_s^{t,x,v} - v| >\underline r_0\right\}
		=\PP\left\{\max_{0 \leq s \leq t} |V_s^{t,x,v} - v|^2 > \underline r_0^2\right\}
		\leq \underline r_0^{-2} \E\left[ \max_{0 \leq s \leq t} |V_s^{t,x,v} - v|^2\right].
\]
In order to bound the term on the right hand side, we argue as follows.   Use Doob's inequality and the Ito Isometry to obtain
\[\begin{split}
	\E\left[ \max_{0 \leq s \leq t} |V_s^{t,x,v} - v|^2\right]
		&\lesssim \E\left[|V_t^{t,x,v} - v|^2\right]
		= \E\left[ \left| \int_0^t \bar\sigma_{R,\eps}(X_s^{t,x,v}, V^{t,x,v}_s) dB_s\right|^2\right]\\
		&= \E\left[\int_0^t |\bar\sigma_{R,\eps}(X_s^{t,x,v}, V_s^{t,x,v})|^2 \dd s \right].
\end{split}\]
Using the asymptotics of $\bar \sigma_{R,\eps}$ and the fact that $\eps < 1$,
we now obtain bounds on $\E[\max_s |V_s^{t,x,v}-v|^2]$. 
It is useful to consider two cases separately:

First, if $\gamma \in [-3,-2]$, then  $\bar \sigma_{R,\eps}$ is bounded above independently of $X_s^{t,x,v}$ and $V_s^{t,x,v}$.  Hence, we see that
\begin{equation}\label{eq:c110}
	\E\left[ \max_{0 \leq s \leq t} |V_s^{t,x,v} - v|^2\right]
		\leq C_0 t,
\end{equation}
where $C_0$ represents the implied constant above and the upper bound on $\bar \sigma_{R,\eps}$.  In this case,
\[
	\PP\left\{\max_{0 \leq s \leq t} |V_s^{t,x,v} - v| >\underline r_0\right\}
		\leq \frac{C_0t}{\underline r_0^2}.
\]
Letting $T_2 = \underline r_0^2/2C_0$ the above is bounded by $1/2$.  Combining this with~\eqref{eq:chris1} and~\eqref{eq:chris2}, we obtain, for all $(t,x,v) \in [0,T_*]\times B_{\underline r_0}(x_0)\times B_{\underline r_0}(v_0)$,
\[
	f_{R,\eps}(t,x,v) \geq \frac{\delta_0}{2}.
\]
From our definitions of $T_*$ and $T_2$, it is clear that $T_* = T$ if $r_0$ is sufficiently large, depending only on $|v_0|$.  This finishes the proof of the claim in the case $\gamma \leq -2$.

On the other hand, if $\gamma \in(-2,0)$ then
\[\begin{split}
	\E\left[ \max_{0 \leq s \leq t} |V_s^{t,x,v} - v|^2\right]
		&\lesssim \E\left[\int_0^t (1 + |V_s^{t,x,v}|^{1 + \gamma/2})^2 \dd s \right]
		\lesssim t \E\left[1 + \max_{0\leq s \leq t} |V_s^{t,x,v} - v|^{2+\gamma} + |v|^{2+\gamma}\right].
\end{split}\]
Above, we used Lemma~\ref{l:coeffs} in the first inequality.  We use Young's inequality on the second term in the expectation, obtaining
\[
	\E\left[ \max_{0 \leq s \leq t} |V_s^{t,x,v} - v|^2\right]
		\leq C_0t \left(1 + (|v_0|+r_0)^{2+\gamma} \right) + (C_0 t)^\frac{2}{|\gamma|} + \frac{1}{2} \E\left[\max_{0\leq s \leq t} |V_s^{t,x,v} - v|^2\right],
\]
where $C_0$ again comes from the implied constant above.  After increasing $C_0$, this may be re-arranged to give
\begin{equation}\label{eq:c111}
	\E\left[ \max_{0 \leq s \leq t} |V_s^{t,x,v} - v|^2\right]
		\leq C_0t \left(1 + (|v_0|+\underline r_0)^{2+\gamma} \right) + (C_0 t)^\frac{2}{|\gamma|}.
\end{equation}
Hence, we obtain
\begin{equation}\label{eq:c112}
	\PP\left\{\max_{0 \leq s \leq t} |V_s^{t,x,v} - v| >\underline r_0\right\}
		\leq \frac{C_0T_2}{\underline r_0^{2}} \left(1 + (|v_0|+\underline r_0)^{2+\gamma} \right) + \frac{(C_0 T_2)^\frac{2}{|\gamma|}}{\underline r_0^{2}}.
\end{equation}
We now let
\[
	T_2 = \frac{1}{4} \min\left\{\frac{\underline r_0^2}{C_0 (1 + (|v_0| + \underline r_0)^{2+\gamma})}, \frac{\underline r_0^{|\gamma|}}{C_0}\right\}.
\]
Then the right hand side of~\eqref{eq:c112} is smaller than $1/2$.  Using this along with~\eqref{eq:chris1} and~\eqref{eq:chris2}, we once again obtain that, for all $(t,x,v) \in [0,T_*]\times B_{\underline r_0}(x_0)\times B_{\underline r_0}(v_0)$,
\[
	f_{R,\eps}(t,x,v) \geq \frac{\delta_0}{2}.
\]
Using the definitions of $T_*$ and $T_2$, we again note that $T_* = T$ if $r_0$ is sufficiently large.

Our bounds do not depend on $R$ and $\eps$.  Hence, taking $R\to\infty$ and $\eps \to 0$, we obtain
\begin{equation}\label{eq:mass_core_time2}
		f(t,x,v) \geq \frac{\delta_0}{2} \1_{B_{\underline r_0}(x_0)\times B_{\underline r_0}(v_0)}(x,v),
\end{equation}
which finishes the proof of the Step 1.


\medskip

\subsubsection*{Step 2: Spreading mass in $v$ for $x\sim x_0$:}
The next step is to show that the mass of $f$ instantaneously ``spreads out'' in $v$.  We require this
to ``spread out'' the mass in $x$ in Step 3, below.  We make this explicit.
Fix $0 < \underline T \leq t \leq T_*$.  For any $x \in B_{\underline r_0/2}(x_0)$, 
\begin{equation}\label{eq:step2}
	f(t,x,\cdot) \geq \nu \exp\left\{ - \rho |v|^{\max\{4,3-\gamma\}}\right\},
\end{equation}
where $\e$ and $\rho$ are as in the statement of the proposition.

Applying \Cref{l:lower} and using the definition of $\sigma_{R,\eps}$, we find that there exists $\lambda_0$, depending only on $\delta_0$ and $r_0$ such that
\begin{equation}\label{eq:ellipticity_bound}
	\bar \sigma_{R,\eps}(t,x,v) \geq \lambda_0 \vv^{\gamma/2}.
\end{equation}


Fix any $R > 2(|v| + |v_0| + r_0)$ and any $\eps \in (0,1)$.   Let $t_c =  \underline r_0 / (4\max\{|v|,|v_0|,N_{\underline r_0}|v|^{2-\gamma}\})$, for $N_{\underline r_0}\geq 1$ to be determined.  We first prove the claim when $t < t_c$.  Then we let $\tau_{2|v|} = \inf\{s>0 : |V_s^{t,x,v}|\geq 2\max\{|v|,|v_0|,1\}\}$ and notice that
\[
	f_{R,\eps}(t,x,v)
		\geq \E[1_{\{\tau_{2|v|} > t\}} f_{in}(X_s^{t,x,v},V_s^{t,x,v})]
		\geq \frac{\delta_0}{2} \PP\left\{\tau_{2|v|}>t, (X_t^{t,x,v},V_t^{t,x,v})\in B_{\underline r_0}(x_0)\times B_{\underline r_0}(v_0)\right\}.
\]
Since $t< t_c$ then it follows that, if $t < \tau_{2|v|}$, $X_t^{t,x,v} \in B_{\underline r_0}(x_0)$.  Hence the above simplifies to
\begin{equation}\label{eq:c113}
	f_{R,\eps}(t,x,v)
		\geq \frac{\delta_0}{2} \PP\left\{\tau_{2|v|}>t, V_t^{t,x,v}\in B_{\underline r_0}(v_0)\right\}.
\end{equation}
Define $\bar v: [0,t] \to \R^d$ as $\bar v(s) = v + (s (v_0-v))/t$.  Then~\eqref{eq:c113} further reduces to
\begin{equation}\label{eq:c114}
	f_{R,\eps}(t,x,v)
		\geq \frac{\delta_0}{2} \PP\left\{ \max_{0\leq s \leq t} |V_s^{t,x,v} - \bar v(s)| \leq \underline r_0\right\}.
\end{equation}

In order to obtain a lower bound on the right hand side of~\eqref{eq:c114}, we use Girsanov's transform to
change probability measures to $\Q$ such that $Y_s := V_s^{t,x,v} - \bar v(s)$ and
\[
	\dd Y_s = \bar\sigma_{R,\eps}(X_s^{t,x,v}, V_s^{t,x,v}) \dd \hat B_s
\]
where $\hat B_s$ is a $\Q$-Brownian motion. Let $A_{\underline r_0} = \{\max_{0\leq s\leq t} |Y_s| \leq \underline r_0\}$.  Then we have that
\[\begin{split}
	\PP\left\{ \max_{0\leq s \leq t} |V_s^{t,x,v} - \bar v(s)| \leq \underline r_0\right\}
		&= \E_\Q\left[\exp\left\{ - \int_0^t \dot{\bar v}  \bar \sigma_{R,\eps}^{-1} \dd \hat B_s  - \frac{1}{2}\int_0^t |\dot{\bar v} \sigma_{R,\eps}^{-1}|^2 \dd s \right\} \1_{A_{\underline r_0}}\right]\\
		&\geq e^\frac{-|v_0-v|^2}{C_0t\lambda_0^2(1+|v_0|+|v|+\underline r_0)^{\gamma}}\E_\Q\left[\exp\left\{ - \int_0^t \dot{\bar v}  \bar \sigma_{R,\eps}^{-1} \dd \hat B_s\right\} \1_{A_{\underline r_0}}\right].
\end{split}
\]
Here $C_0$ is the implied constant in~\eqref{eq:ellipticity_bound}. 
Let $q_{\underline r_0} = \Q(A_{\underline r_0})$.  Then, using Jensen's inequality, we obtain
\begin{equation}\label{eq:c117}
\begin{split}
	\PP\left\{ \max_{0\leq s \leq t} |V_s^{t,x,v} - \bar v(s)| \leq \underline r_0\right\}
		&\geq e^\frac{-|v_0-v|^2}{C_0t\lambda_0^2(1+|v_0|+|v|+\underline r_0)^{\gamma}} q_{\underline r_0} \E_\Q\left[ \1_{A_{\underline r_0}} q_{\underline r_0}^{-1} \exp \int_0^t \dot{\bar v} \bar \sigma_{R,\eps}^{-1} \dd \hat B_s \right]\\
		&\geq e^\frac{-|v_0-v|^2}{C_0t\lambda_0^2(1+|v_0|+|v|+\underline r_0)^{\gamma}} q_{\underline r_0} \exp \E_\Q\left[\1_{A_{\underline r_0}} q_{\underline r_0}^{-1} \int_0^t \dot{\bar v} \sigma_{R,\eps}^{-1} \dd \hat B_s\right].
\end{split}
\end{equation}
We note that, after fixing a sufficiently large $N_{\underline r_0}$ depending only on $|v_0|$, $\underline r_0$, and the physical quantities~\eqref{eq:physical}, we may obtain a lower bound for $q_{\underline r_0}$ by arguing exactly as in Step 1.  On the other hand,  letting $\tilde R = |v| + |v_0| + \underline r_0$, 
\begin{equation}\label{eq:c116}
\begin{split}
	\left|\E_\Q\left[ \1_{A_{\underline r_0}} \int_0^t \dot{\bar v} \sigma_{R,\eps}^{-1} \dd \hat B_s\right]\right|
		&= \left|\E_\Q\left[ \1_{A_{\underline r_0}} \int_0^t \dot{\bar v} \sigma_{\tilde R,\eps}^{-1} \dd \hat B_s\right]\right|
		\leq \Q(A_{\underline r_0})^{1/2} \E_\Q\left[ \1_{A_{\underline r_0}} \int_0^t |\dot{\bar v} \sigma_{\tilde R,\eps}^{-1}|^2 \dd s\right]^{1/2}\\
		&\leq q_{\underline r_0}^{1/2} \frac{|v_0-v|}{C_0\sqrt t\lambda_0 (1 + |v_0|+|v|+\underline r_0)^{\gamma/2}}
		\leq q_{\underline r_0} + \frac{|v_0-v|^2}{C_0^2 t \lambda_0^2 (1 + |v_0| + |v| + \underline r_0)^\gamma}
\end{split}
\end{equation}
where we used H\"older's inequality and the Ito isometry in the first inequality, and we used the lower bound~\eqref{eq:ellipticity_bound} in the second inequality.

Let $\beta = \max\{3-\gamma, 4\}$.  Combining the discussion regarding $q_{\underline r_0}$ with~\eqref{eq:c116},~\eqref{eq:c117}, and~\eqref{eq:c114}, and using the fact that $|v|^{2-\gamma} t^{-1} \leq |v|^{\beta} + C_t$ for some $C_t$ depending only on $t$ and $\gamma$, we obtain 
\[
	f_{R,\eps}(t,x,v)
		\geq \nu \exp\left\{ - \rho \frac{|v|^{2-\gamma}}{t}\right\}
		\geq \nu' \exp\left\{ - \rho' |v|^{\beta}\right\},
\]
for some $\nu'$ and $\rho'$ with the same dependences of $\nu$ and $\rho$. Since all estimates are independent of $R$ and $\eps$, we may conclude the proof of Step 2 in this case by taking the limit as $R\to \infty$ and $\eps \to 0$.

If $t \geq t_c$, we may simply translate the argument in time and use the semi-group property.  Hence,
\[
	f_{R,\eps}(t,x,v)
		\geq \nu \exp\left\{ - \rho \frac{|v|^{2-\gamma}}{t_c}\right\}
		\geq \nu' \exp\left\{ - \rho' |v|^{\beta}\right\}.
\]
for some $\nu'$ and $\rho'$ with the same dependencies as $\nu$ and $\rho$.  We conclude exactly as above.  This establishes~\eqref{eq:step2}.

\medskip

\subsubsection*{Step 3: Spreading mass in $x$ for select velocities:}
We now obtain a lower bound on $f$ for all $t$ and $x$ and some subset of velocities.  Specifically, we aim to prove that, for all $t \in (0,T_*]$ and all $x \in \R^3$ there exist $v_{t,x} \in \R^3$, depending only on $|v_0|$, $r_0$, $t$, and $|x-x_0|$, and $\delta_{t,x} > 0$ depending on the same quantities and also $\delta_0$, such that, for all $v\in B_{\underline r_0/4t}(v_{t,x})$,
\begin{equation}\label{eq:small_time_bound1}
	f(t,x,v) \geq \delta_{t,x}.
\end{equation}
Further, $v_{t,x}$ and $\delta_{t,x}$ depend continuously on $(t,x) \in (0,T_*]\times \R^3$. 
In particular, this provides a lower bound of the form~\eqref{eq:mass_core} for $f(t,\cdot,\cdot)$ at $(x,v_{t,x})$ for any $t>0$ and $x \in \R^3$. 

To establish~\eqref{eq:small_time_bound1}, fix any $t \in(0,T_*]$ and $R>0$ to be determined.  Again notice that up to shifting in time, we may assume that $t< t_0$, for $t_0$ to be determined below, and that the lower bound from Step 2~\eqref{eq:step2} holds for the initial data. 
Define $v_{t,x} = -(x_0 -x)/t$.   Fix any $\eps \in (0,1)$ and any $R \geq 2(|v_{t,x}| + \underline r_0/2t + 1)$, and let 
\[
	\tau_{\underline r_0/t} = \inf\left\{s : |V_s^{t,x,v} - v_{t,x}| \geq \frac{\underline r_0}{4t}\right\}.
\]
Fix any $v \in B_{\underline r_0 /4t}(v_{t,x})$ and, using~\eqref{eq:step2}, there exists $\tilde m_0>0$ depending on $\delta_0$, $|v_0|$, $r_0$, $t$, and $|x-x_0|$ such that
\begin{equation}\label{eq:chris5}
\begin{split}
	f_{R,\eps}(t,x,v)
		&\geq \tilde m_0 \E[ \1_{\{\tau_{\underline r_0/t} > t\}}\1_{\left\{X_t^{t,x,v} \in B_{\underline r_0/2}(x_0)\right\}} \1_{\left\{V_{t}^{t,x,v} \in B_{\underline r_0/2t}(v_{t,x})\right\}}]
		= \tilde m_0 \PP\left\{\tau_{\underline r_0/t} > t\right\}.
\end{split}
\end{equation}
In the equality above, we used that if $\tau_{\underline r_0/t} > t$ then $V_t^{t,x,v} \in B_{\underline r_0/2t}$ and
\[
	|X_t^{t,x,v}  - x_0|= \left|x - x_0 + \int_0^t V_s^{t,x,v} ds\right|
		= \left|\int_0^t (V_s^{t,x,v} - v_{t,x}) ds\right|
		< \frac{\underline r_0}{2}.
\]

Hence, in order to finish the proof of Step 3, we need only obtain a lower bound on
\begin{equation}\label{eq:c119}
	\PP\{\tau_{\underline r_0/t} > t\}
		= \PP\left\{ \max_{0\leq s \leq t} |V_s^{t,x,v} - v_{t,x}| \leq \frac{\underline r_0}{2t}\right\}
		= 1 - \PP\left\{ \max_{0\leq s \leq t} |V_s^{t,x,v} - v_{t,x}| > \frac{\underline r_0}{2t}\right\}.
\end{equation}
Since the estimates are significantly simpler when $\gamma \leq -2$, we only show the case when $\gamma > -2$.  Indeed, we follow the work in Step 1 and use the estimate~\eqref{eq:c111} to obtain
\begin{equation}\label{eq:c115}
	\PP\left\{ \max_{0\leq s \leq t} |V_s^{t,x,v} - v_{t,x}| > \frac{\underline r_0}{2t}\right\}
		\leq \frac{t^2}{\underline r_0^2}\left(C_0 t \left(1 + \left|\frac{x-x_0}{t}\right| + \frac{\underline r_0}{t}\right)^{2+\gamma} + C_\gamma (C_0 t)^\frac{2}{|\gamma|}\right),
\end{equation}
where $C_0$ is a positive universal constant.  We are now in a position to define $t_0$.  Let
\[
	t_0 = \frac{1}{4} \min\left\{
		\left(\frac{\underline r_0^2}{3^{2+\gamma}C_0}\right)^\frac13,
		\left(\frac{\underline r_0^2}{3^{2+\gamma}C_0 |x-x_0|^{2+\gamma}}\right)^\frac{1}{|\gamma|},
		\frac{\underline r_0}{3^\frac{2+\gamma}{|\gamma|}C_0^\frac{1}{|\gamma}},
		\left(\frac{\underline r_0^2}{C_\gamma C_0^\frac{2}{|\gamma|}}\right)^\frac{|\gamma|}{2 + |\gamma|} \right\},
\]
and then~\eqref{eq:c115} implies that
\[
	\PP\left\{ \max_{0\leq s \leq t} |V_s^{t,x,v} - v_{t,x}| > \frac{\underline r_0}{2t}\right\}
		\geq \frac{1}{2}.
\]
Plugging this into~\eqref{eq:c119} and then~\eqref{eq:chris5} yields the lower bound for $f_{R,\eps}$.  Taking the limits $R\to \infty$ and $\eps \to 0$ yields the lower bound for $f$.  The continuity of this lower bound and of $v_{t,x}$ is clear from the proof.  This finishes the proof of Step 3.

\medskip

\subsubsection*{Step 4: Spreading mass in $v$ for all $x$:}

Since the lower bound of $f$ obtained in Step 3 holds locally uniformly in $x$, we may repeat the arguments of Steps 1 and 2 in order to see that, for any $(t,x,v) \in (0,T_*]\times \R^3\times\R^3$, there exists $m(t,x,v)>0$, depending only on $t$, $|v|$, $|v_0|$, $r_0$, $|x-x_0|$, and $\delta_0$, such that $f(t,x,v) \geq m(t,x,v)$.  This finishes the proof.
\end{proof}

\subsection{Using \Cref{prop:mass_weak} to obtain \Cref{thm:mass}}

\begin{proof}[Proof of \Cref{thm:mass}.(i)]
First, note that the assumed regularity of $f$ implies that $f(0,\cdot,\cdot)$ is H\"older continuous (see the discussion after formula \eqref{e:nondivergence_regularized}) so that \eqref{eq:mass_core} is satisfied 
for some $\delta_0$, $r_0$, $x_0$, $v_0$. Next, notice that applying \Cref{prop:mass_weak} one time implies
that $f$ is positive everywhere for some small time interval $(0,T_*]$.  At which point, we may re-apply
\Cref{prop:mass_weak} on the time interval $(T_*,T]$.  This is possible because, by choosing $\delta_{T_*}$
small enough, we may find $r_{T_*}$ arbitrarily large  
such that
\begin{equation*}
	\delta_{T_*} \1_{B_{r_{T_*}}(0)\times B_{r_{T_*}}(0)}(x,v) \leq f(T_*,x,v).
\end{equation*}
The sub-Gaussian lower bound then follows directly from \Cref{prop:mass_weak}, concluding the proof of \Cref{thm:mass}.(i).
\end{proof}

\begin{proof}[Proof of \Cref{thm:mass}.(ii)]
Arguing as above, and using the well-distributed initial data, we immediately obtain $\delta_1$, depending only on $R$, $\delta$, $r$, $\underline T$, $\overline T$, and the quantities in~\eqref{eq:physical} such that, for all $(t,x,v) \in [\underline T/2, T]\times \R^3 \times B_1(0)$, $f(t,x,v) \geq \delta_1.$ 
Define 
\[\underline f(t,x,v) = \delta_1 \exp\{- \beta(t) |v|^{2-\gamma}\},\] for $\beta: \R_+ \to \R_+$ to be determined.  We claim that $\underline f$ is a sub-solution to the linear Landau equation; that is, letting $L:= \partial_t + v\cdot \nabla_x - \tr(\bar a[f] D^2_v \cdot) - \bar c[f]$, we claim that $L \underline f \leq 0$ in $[\underline T,T]\times \R^3\times \{|v|\geq 1\}$.

By a direct computation, we obtain
\[\begin{split}
	L\underline f
		&= \underline f \left[
			- \beta' |v|^{2-\gamma} - \bar a_{ij}[f]\left((2-\gamma)^2 \beta^2 |v|^{-2\gamma} v_iv_j - \beta (2-\gamma) |v|^{-\gamma} (\delta_{ij} - \gamma |v|^{-2} v_iv_j)\right)
			- \bar c[f]
		\right]\\
		&\leq \underline f \left[
			 -  \beta' |v|^{2-\gamma} - C^{-1} \beta^2 |v|^{2-\gamma} + C \vv^2.
		\right]
\end{split}\]
where $C$ is a constant depending only on $\delta_1$ and the physical quantities~\eqref{eq:physical}.  In the inequality we used the anisotropic upper bounds for $\bar a[f]$ of \cite[Appendix A]{henderson2017smoothing}, the lower bounds of
 \Cref{l:lower}, and the fact that $\bar c[f] \geq 0$.

At this point, we choose $\beta(t) = 1+C_1/(t-\underline T/2)$. 
We choose the constant $C_1$ large enough that $-\beta'(t) |v|^{2-\gamma} + C\vv^2\leq  C^{-1}\beta^2(t)|v|^{2-\gamma}$ for all $|v|\geq 1$. 
Hence, $L\underline f \leq 0$. Since $f(t,x,v) \geq \delta_1 \geq \underline f(t,x,v)$ whenever $t\in [\underline T/2,T]$ and $|v| = 1$, and we can extend $\underline f$ smoothly by zero when $t= \underline T/2$ and $|v|\geq 1$, we have $\underline f \leq f$ on the parabolic boundary of $[\underline T/2, T]\times \R^3 \times \{|v|\geq 1\}$. It follows from the comparison principle applied to $f$ and $\underline f$ that $\underline f \leq f$ in $[\underline T/2, T]\times \R^3 \times \{|v|\geq 1\}$.  This concludes the proof.
\end{proof}

\subsection{Optimality of self-generating lower bounds}\label{s:A}

We show that the asymptotic behavior of the lower bounds in Theorem \ref{thm:mass}(ii) cannot be improved in general.  We find quite general initial data such that corresponding upper bounds hold.  This upper bound may perhaps be known in the space homogeneous setting, but we are unable to find a reference.
\begin{proposition}\label{p:upper_Gaussian}
	Let $T>0$ and $f$ be any solution of~\eqref{e:divergence} on $[0,T]\times\R^3\times\R^3$ such that the physical quantities~\eqref{eq:physical} remain bounded on $[0,T]$ and such that, for some $\rho, K>0$,
	\[
		f(0,x,v) \leq K \exp\{- \rho |v|^{2-\gamma}\}.
	\]
	Then there exists $\alpha$ and $C$, depending only on $\gamma$ and the  quantities~\eqref{eq:physical} such that
	\[
		f(t,x,v) \leq K \exp\left\{\alpha t - \frac{\rho |v|^{2-\gamma}}{2C\rho t + 1}\right\} \qquad \text{ for all } (t,x,v) \in [0,T]\times \R^3\times\R^3.
	\]
\end{proposition}

Before beginning the proof, we note that there exist $f$ satisfying the conditions of~\Cref{p:upper_Gaussian}, see, for example,~\Cref{t:LWP}.

\begin{proof}
In order to conclude, we need only construct a super-solution in $[0,T]\times \R^3\times \R^3$.  Let $L$ be as in the proof of \Cref{thm:mass}.(ii).  Let $\overline f(t,x,v) = K e^{\alpha t - \beta(t) |v|^{2-\gamma}}$, where $\alpha>0$ and $\beta:\R_+ \to \R_+$ are to be determined. 
Computing $L\overline f$ directly as in the proof of Theorem \ref{thm:mass}.(ii), and using the upper bounds for $\bar a[f]$ and $\bar c[f]$ in \cite[Appendix A]{henderson2017smoothing} and the fact that $\bar a[f]$ is nonnegative definite, we find a constant $C$, depending only on $\gamma$ and the quantities~\eqref{eq:physical}, such that
\[
	L\overline f \geq \overline f \left[
			\alpha -  \beta'(t) |v|^{2-\gamma} - C\beta^2(t) \vv^{2-\gamma}  - C\vv^\gamma
		\right].
\]
Then, we define $\beta(t) = \rho/(2\rho C t + 1)$, so that the positive term $-\beta'(t) |v|^{2-\gamma}$ dominates for large $|v|$. Choosing $\alpha >0$ large enough, we have $L\overline f \geq 0$ for small $|v|$ as well. 
By our assumption on the initial data, the comparison principle implies that $f \leq \overline f$ on $[0,T]\times \R^3\times\R^3$.  This concludes the proof.
\end{proof}


\section{Smoothing and continuation of solutions}\label{s:app}


We are now ready to show that our solutions to \eqref{e:divergence} are $C^\infty$ in all three variables. 

\begin{proof}[Proof of \Cref{thm:smooth}]
We will apply the main theorem of \cite{henderson2017smoothing}. The Gaussian decay of $f$ in $v$ (which is uniform in $t$ and $x$) implies uniform upper bounds on $M(t,x)$, $E(t,x)$, and $H(t,x)$. Theorem \ref{thm:mass} implies that $M(t,x) \geq m_{t,x}>0$ for each $t\in [0,T]$, $x\in \R^3$, and that a uniform, positive lower bound on $M(t,x)$ holds in any cylinder $Q_r(t_0,x_0) = (t_0-r^2]\times B_r(x_0)$ so long as $r^2 < t_0$.  As written, the smoothing theorem \cite[Theorem 1.2]{henderson2017smoothing} 
requires a lower bound on $M(t,x)$ that is uniform in $t$ and $x$. 
However, these proofs are entirely local in $t$ and $x$, and still go through with our locally uniform lower bound on $M(t,x)$.  
Hence, by the (locally uniform) lower bound on the mass $M(t,x)$, Gaussian decay in $v$ of $f$, and the upper bounds on the physical quantities~\eqref{eq:physical} and $H(t,x)$, we conclude $f$ is in $C^\infty([0,T]\times \R^3\times \R^3)$ via \cite[Theorem 1.2]{henderson2017smoothing}.  Since $f$ has uniform Gaussian decay in $v$, the proof in \cite{henderson2017smoothing} shows that all partial derivatives have Gaussian decay which is locally uniform in $t$ and $x$, with constants depending on the order of the derivative. This implies the moment bounds in the statement of the theorem.

 If, in addition, $f_{in}$ is well-distributed, the lower bound on $M(t,x)$ is uniform on $[\underline T/2, T]\times \R^3$ for any $\underline T\in (0,T]$. Together with the uniform Gaussian decay, this implies, via \cite[Theorem 1.2]{henderson2017smoothing} applied to $f(\underline T/2 + t, x, v)$, that all partial derivatives of $f$ satisfy Gaussian-in-$v$ estimates that are uniform on $[\underline T, T]\times \R^3$. 
\end{proof}

Finally, we show that solutions can be extended so long as they are well-distributed initially (see Definition \ref{d:well}) and the hydrodynamic quantities remain bounded. The following is a more precise statement of Theorem \ref{t:extension}:

\begin{theorem}\label{t:extension2}
	Suppose that the assumptions of \Cref{thm:smooth} hold for some $\rho >0$ and $T \in (0, \overline T]$ where $\overline T > 0$. 
	Then there exists $T_1>0$, depending only on $\gamma$, $\rho$, and $\|e^{\rho\vv^2} f(T,\cdot,\cdot)\|_{H^4_{\ul}}$, such that $f$ can be extended to be a solution of~\eqref{e:divergence} in $Y_{T+T_1}^4$. 
	
	If, in addition, the initial data satisfies $e^{\rho_0\vv^2}f_{in}\in H_{\ul}^4$ for some $\rho_0\geq \rho$, and $f_{in}$ is well-distributed with parameters $R$, $\delta$, $r$,  then the solution $f$ satisfies $e^{\min(\mu,\rho_0/2)\vv^2}f \in Y_T^4$ for some $\mu>0$. The constant $\mu$ and the time of extension $T_1$ depend only on $\overline T$, $\gamma$, $R$, $\delta$, $r$, $\|e^{\rho_0\vv^2}f_{in}\|_{H^k_{\ul}}$, and the bounds on \eqref{eq:physical}. 
	In particular, $T_1$ may be chosen independently of $\rho$ and $\|e^{\rho\vv^2} f(T,\cdot,\cdot)\|_{H^k_{\ul}}$.
\end{theorem}

\begin{remark}
	The significance of the decay rate $\min(\mu,\rho_0/2)$ is as follows: in constructing a solution $f$ in Theorem \ref{t:LWP}, our first step was to divide $f$ by $e^{-(\rho_0-\kappa t)\vv^2}$ for a positive constant $\kappa$ that is more or less arbitrary, but the resulting time of existence $T$ depends heavily on $\rho_0/\kappa$ and $\kappa$.  This theorem allows us to remove this dependence when the initial data is well-distributed.
\end{remark}

\begin{proof}


For a solution $f$ with $e^{\rho\vv^2}f \in Y_T^4$, we of course have $e^{\rho\vv^2} f(T,\cdot,\cdot) \in H^k_{\ul}$, and we may apply Theorem \ref{t:LWP} to obtain a solution on $[T,T+T_1]$ for some $T_1$ depending on $\rho$, $\gamma$, and $\|e^{\rho\vv^2}f(T,\cdot,\cdot)\|_{H^k_{\ul}}$. We concatenate this solution with $f$ to obtain a solution on $[0,T+T_1]$.

Next, suppose that $f_{in}$ is well-distributed with parameters $R$, $\delta$, $r$, and that $e^{\rho_0\vv^2}f_{in} \in H^4_{\ul}$. By the uniqueness in Theorem \ref{t:LWP}, there is some $t_0\in(0,T)$ depending on $\|e^{\rho_0\vv^2}f_{in}\|_{H^4_{\ul}}$ such that $e^{\rho_0\vv^2/2}f \in Y_{t_0}^4$. 
We want to apply \cite[Theorem~1.2]{henderson2017smoothing} to show $f$ and its derivatives up to order $4$ have Gaussian decay up to time $T$, with constants as in the statement of the current theorem. As written, \cite[Theorem~1.2]{henderson2017smoothing} requires a uniform upper bound on $H(t,x)$ and a uniform lower bound on $M(t,x)$, but the only place in \cite{henderson2017smoothing} where these two assumptions play a role is in showing a lower ellipticity bound for $\bar a[f]$ of the form \eqref{e:a_lower}. Since $f_{in}$ is well-distributed, Theorem \ref{thm:mass} and Lemma \ref{l:lower} imply such an ellipticity bound holds, with constants depending only on $t_0$, $\overline T$, $\gamma$, the quantities in \eqref{eq:physical}, $R$, $\delta$, and $r$. With this modification, we can apply \cite[Theorem~1.2]{henderson2017smoothing}\footnote{The Gaussian bounds of $f$ were originally proved in~\cite{cameron2017landau} for the case $\gamma \in (-2,0)$.  They were extended to the case $\gamma \in [-3,-2]$ in \cite{henderson2017smoothing} and to apply to derivatives of $f$ via Schauder estimates.} and conclude that there exist positive constants $C$ and $\mu$, depending only on $t_0$, $\overline T$, $\gamma$, $R$, $\delta$, $r$, and the upper bounds on \eqref{eq:physical}, such that
\[
\|e^{\min(\mu,\rho_0/2) \vv^2} f(t,\cdot,\cdot)\|_{H^4_{\ul}} \leq C, \quad t\in [t_0,T].
\]
Now we may proceed as in the first paragraph of the proof and obtain a solution on $[0,T+T_1]$, with $T_1$ as in the statement of the theorem.
\end{proof}

\bibliographystyle{abbrv}
\bibliography{landau_paper}

\end{document}